\documentclass[a4paper, reqno, 11pt]{amsart}

\usepackage[english]{babel}
\usepackage{amsmath}
\usepackage{mathrsfs}
\usepackage{amssymb}
\usepackage{amsthm}
\usepackage{enumerate}
\usepackage{ifthen}
\usepackage{bbm}
\usepackage{color}
\usepackage{xcolor}
\usepackage{graphicx}
\provideboolean{shownotes} 
\setboolean{shownotes}{true}
\usepackage{hyperref}
\usepackage{amsfonts}
\usepackage{listings}
\usepackage{amssymb, bm}
\usepackage{amsbsy}
\usepackage{stmaryrd}
\usepackage{mathtools}
\usepackage{paralist}
\usepackage{enumitem}

\usepackage{float}   
\usepackage{fancybox}		  
\usepackage{verbatim}

\def\XXint#1#2#3{{\setbox0=\hbox{$#1{#2#3}{\int}$ }
\vcenter{\hbox{$#2#3$ }}\kern-.6\wd0}}

\usepackage{geometry}
\newcommand{\margnote}[1]{
\ifthenelse{\boolean{shownotes}}%
{\marginpar{\raggedright\tiny\texttt{#1}}}%
{}%
}

\newcommand{\hole}[1]{
\ifthenelse{\boolean{shownotes}}%
{\begin{center} \fbox{ \rule {.25cm}{0cm}
\rule[-.1cm]{0cm}{.4cm} \parbox{.85\textwidth}{\begin{center}
\texttt{#1}\end{center}} \rule {.25cm}{0cm}}\end{center}}
{}
}
\newtheorem{thm}{Theorem}[section]

\newtheorem{prop}[thm]{Proposition}
\newtheorem{lem}[thm]{Lemma}
\newtheorem{cor}[thm]{Corollary}
\newtheorem{rem}[thm]{Remark}

\theoremstyle{definition}
\newtheorem{defn}[thm]{Definition}

\newcommand{\e}{\varepsilon}		       
\newcommand{\R}{\mathbb{R}}

\newcommand{\N}{\mathbb{N}}

\newcommand{\E}{\mathbb{E}}
\newcommand{\Lip}{\mathrm{Lip}}
\newcommand{\dive}{\mathop{\mathrm {div}}}

\newcommand{\weakto}{\rightharpoonup}
\newcommand{\weaktos}{\stackrel{*}{\rightharpoonup}}
\newcommand{\de}{\,\mathrm{d}}

\newcommand{\supp}{\mathrm{supp}}

\newcommand{\Pb}{\mathscr{P}}

\newcommand{\Pac}{\mathscr{P}^\mathrm{ac}}
\newcommand{\Pc}{\mathscr{P}_c(\R^d)}
\newcommand{\Pn}{\mathscr{P}_2(\mathbb{R}^d)}
\newcommand{\Pp}{(\mathcal{P})}
\newcommand{\Ppe}{(\mathcal{P}_\e)}

\newcommand{\U}{\mathcal{U}}
\newcommand{\Law}{\mathcal{L}}
\newcommand{\FR}[1]{{\color{blue} #1}}

\newcommand{\schema}[1]{{\bf \sc #1}}
\newenvironment{pschema}[1]{\vspace{3mm}
\noindent\begin{Sbox}\begin{minipage}{.95\columnwidth}\vspace{2mm}\begin{center}{\large \schema{#1}}\vspace{5mm}\\
\begin{minipage}{0.9\textwidth}}{\end{minipage}\end{center}\vspace{2mm}\end{minipage}\end{Sbox}\fbox{\TheSbox}\vspace{3mm}}

\numberwithin{equation}{section}

\subjclass{49J45, 93E20, 49J20.}

\keywords{Mean-field equations, optimal control, vanishing viscosity.}

\begin{document}
\title[Vanishing viscosity in mean-field optimal control]{Vanishing viscosity in mean-field optimal control}

\author[G. Ciampa]{Gennaro Ciampa}
\address[G.\ Ciampa]{Dipartimento di Matematica ``Federigo Enriques", Universit\`a degli Studi di Milano, Via Cesare Saldini 50, 20133 Milano, Italy.}
\email[]{\href{gciampa@}{gennaro.ciampa@unimi.it}}

\author[F. Rossi]{Francesco Rossi}
\address[F. Rossi]{Dipartimento di Matematica ``Tullio Levi Civita''\\ Universit\`a degli Studi di Padova\\Via Trieste 63 \\35131 Padova \\ Italy}
\email[]{\href{francesco.rossi@math.unipd.it}{francesco.rossi@math.unipd.it}}

\begin{abstract}
We show the existence of Lipschitz-in-space optimal controls for a class of mean-field control problems with dynamics given by a non-local continuity equation. The proof relies on a vanishing viscosity method: we prove the convergence of the same problem where a diffusion term is added, with a small viscosity parameter.

By using stochastic optimal control, we first show the existence of a sequence of optimal controls for the problem with diffusion. We then build the optimizer of the original problem by letting the viscosity parameter go to zero.
\end{abstract}

\maketitle

\section{Introduction}

In recent years, the study of systems describing crowds of interacting agents has drawn a huge interest from the mathematical and control community.
A better understanding of such interaction phenomena can have a strong impact in several key applications, such as road traffic and egress problems for pedestrians. For a few reviews about this topic, see \textit{e.g.} \cite{axelrod,active1,camazine,CPT,helbing,Jackson2010,sepulchre}.

Mean-field equations are the natural limit of such systems, composed of a large number $N$ of interacting particles, when $N$ tends to infinity. The  state of the system is then a density or, more in general, a measure. Mathematically speaking, the system is often transformed from a large-dimensional ordinary differential equation to a partial differential equation, via the so-called mean-field limit, see e.g. \cite{golse, Sznitman91}.

The finite-dimensional models for interacting agents can either be deterministic (in which the position of each agent is clearly identified), or probabilistic (in which the position of each agent is a probability measure). While deterministic models are based on a (supposedly) perfect knowledge of the dynamics, probabilistic models naturally arise when either individual dynamics or interactions are subjected to some form of noise. As a consequence, mean-field equations have deeply different natures in the two cases: the limit of deterministic models is often a continuity equation, while for probabilistic models it is a diffusion equation. See \cite{CarmonaDelarueI,CarmonaDelarueII,golse} for a comprehensive introduction.\\

Beside the dynamics of mean-field equations, it is now relevant to study control problems for them, that are known as {\it mean-field control problems}. In the mean-field limit for deterministic models, a few articles have been dealing with controllability results \cite{Duprez2019,Duprez2020} or explicit syntheses of control laws \cite{Caponigro2015,ControlKCS}. Most of the literature focused on optimal control problems, with contributions ranging from existence results \cite{LipReg,FLOS,FPR,MFOC} to first-order optimality conditions \cite{bensoussan1,MFPMP,PMPWassConst,SetValuedPMP,PMPWass,Cavagnari2018,Cavagnari2020,Pogodaev2016}, to numerical methods \cite{achdou2,Burger2020}. The linear-quadratic case is studied in \cite{LQ-CDC} for the deterministic setting and in \cite{bensoussan1, bensoussan2} for the probabilistic one.\\

Our aim in this article is to develop one more technique to solve mean-field optimal control problems. Indeed, it is natural to expect that, in finite dimension, (uncontrolled) probabilistic models converge (in some sense) to deterministic ones as the noise decreases to zero. See e.g. \cite{oksendal}. The same holds for (uncontrolled) partial differential equations, when solutions to  advection-diffusion equations converge to solutions of the continuity equation as the noise parameter (also known as viscosity) goes to zero. This is the basic idea of the {\it vanishing viscosity method}, see \cite{BianchiniBressan, Kruzkov}. Observe that, in general, for the partial differential equation with diffusion term, stronger regularity of solution is ensured and better computational methods are available. Moreover, standard numerical methods for non-local transport equations (like for example Lax-Friedrichs scheme, see \cite{numerico}) introduce the so-called numerical viscosity, which represents the numerical counterpart of adding the viscous term to the equation. It is then very desirable to be able to use methods for diffusive equations and then pass to the continuity equation by a vanishing viscosity method.

Our aim is exactly to provide a vanishing viscosity result for mean-field control problems. In our article, we deal with two optimal control problems, corresponding to the deterministic and probabilistic settings. On one side, the deterministic optimal control problem is 

\begin{pschema}{Problem $\Pp$}
Find
\begin{equation*}
    \min_{u\in\mathcal{A}}J(\mu,u),
\end{equation*}
where the cost $J$ is 
\begin{equation}\label{def:costo}
J(\mu,u):=\displaystyle \int_0^T\int_{\R^d} \left(f(t,x,\mu_t)+\psi(u(t,x))\right)\mu_t(\de x) \de t +\int_{\R^d}g(x,\mu_T)\mu_T(\de x),
\end{equation}
and $\mu\in C([0,T];\Pn)$ is a solution of 
\begin{equation}\label{eq:c}
\begin{cases}
\partial_t \mu_t+\dive\left[(b(t,x,\mu_t)+u(t,x))\mu_t\right]=0,\\
\mu|_{t=0}=\mu_0
\end{cases}
\end{equation}
with initial state $\mu_0\in\Pn$ with compact support.

The set of admissible controls is
$$
\mathcal{A}:=L^\infty((0,T);L^1(\R^d, U;\de \mu_t)),
$$
where $U\subset\R^d$. 
\end{pschema}

We now add a viscosity term on the right hand side of \eqref{eq:c}, with $\e>0$ being the diffusion parameter, connected to the viscosity of the system. Then, we consider the following problem:

\begin{pschema}{Problem $\Ppe$}
Take $\Pp$ and replace $\mu$ solution of \eqref{eq:c} with $\mu^\e$ solution of

\begin{equation}\label{eq:ad}
\begin{cases}
\partial_t \mu_t^\e+\dive\left[(b(t,x,\mu_t^\e)+u(t,x))\mu_t^\e\right]=\e \Delta \mu_t^\e,\\
\mu^\e|_{t=0}=\mu_0.
\end{cases}
\end{equation}
Replace the set of admissible controls $\mathcal{A}$ with 
$$
\mathcal{A}^\e:=L^\infty((0,T);L^1(\R^d, U;\de \mu^\e_t)).
$$
\end{pschema}

Under natural hypotheses, both solutions $(\mu,u)$ of the deterministic problem $\Pp$ and solutions $(\mu^\e,u^\e)$ of the probabilistic problem $\Ppe$ exist. In this framework, the natural questions about vanishing viscosity are the following:
 \begin{itemize}
     \item Do we have convergence of optimal controls $u^\e\to u$?
     \item Do we have convergence of optimal trajectories $\mu^\e\to \mu$?
     \item Do we have convergence of costs $J(\mu^\e,u^\e)\to J(\mu,u)$?
 \end{itemize}
 
Such questions do not have a general answer. Our main result states that, under quite natural hypotheses, all answers are positive. 



\begin{thm}\label{thm:main}
Assume the following:
\begin{itemize}
    \item the set of admissible control values $U\subset\R^d$ is convex and compact;
    \item the vector field $b$ is $C^{1,1}$ regular, i.e. {\bf Assumption (B)} in Section \ref{s-assB} below holds; 
    \item the functions $f,\psi,g$ in $J$ are $C^{1,1}$ regular, i.e. {\bf Assumption (J)} in Section \ref{s-assJ} below holds;
    \item the function $\psi$ is $\lambda$-convex, for some $\lambda>0$, 
    i.e. {\bf Assumption (C)} in Section \ref{s-assJ} below holds.
\end{itemize}
Then, there exists $\Lambda>0$, which depends only on the final time $T$ and the Lipschitz constant of the functions $b, f, g$, such that if $\lambda>\Lambda$ there exist:
\begin{itemize}
    \item a unique solution $(\mu^\e,u^\e)\in C([0,T];\Pn)\times L^\infty((0,T);\Lip(\R^d,U))$ of $\Ppe$, for each $\e>0$,
    \item a solution $(\mu,u)\in C([0,T];\Pn)\times L^\infty((0,T);\Lip(\R^d,U))$ of $\Pp$,
\end{itemize}
such that, up to a sub-sequence which we do not relabel, the following convergences hold:
\begin{itemize}
\item[$(i)$] $u^\e\weakto u$ in  $L^2((0,T);W^{1,p}_\mathrm{loc}(\R^d,U))$ for every $1\leq p<\infty$;
\item[$(ii)$] $\mu^\e\to \mu$ in $C([0,T],\Pn)$;
\item[$(iii)$] $J(\mu,u)\leq \liminf_{\e\to 0}J(\mu^\e,u^\e)$.
\end{itemize}
\end{thm}

\begin{rem}
We will see that $\Lambda=O(T)$ and in particular is small for $T$ small. This implies that, at least for small times, the condition $\lambda>\Lambda$ is always satisfied, see Remark \ref{rem:tempi-piccoli} below.
\end{rem}

\begin{rem}
The hypothesis of the theorem are certainly not minimal and they can be relaxed with technicalities. For example, compactness of the support of $\mu_0$ or growth conditions of $f$ and $g$ can be relax, yet we keep them to make the presentation easier. Moreover, beside standard regularity hypotheses, the most interesting and crucial requirement is certainly the strict convexity of the control cost $\psi$. It is worth to note that if a convexity assumption on $f$ and $g$ holds, then the value of $\Lambda$ decreases; we will provide further comments in Remark \ref{rem:convexityfg} below. 
\end{rem}

We introduce in Section 3 a Stochastic Optimal Control problem, called (SOC), as the counterpart of the problem $\Ppe$. The assumption on the constant $\lambda$ guarantees that the contribution of $\psi$ for (SOC) is strongly convex, providing a sufficient condition for optimality for (SOC). 
The latter allows us to construct the optimal control for $\Ppe$. However, one might also ask under which assumptions we can prove the convergence of the vanishing viscosity method if an optimal control of $\Ppe$ is a-priori assigned. We have the following result.

\begin{cor}\label{main-cor}
Assume that there exists an optimal pair $(\mu^{\e},u^{\e})$ for $(\mathcal{P}_{\e})$ and that the following assumptions hold:
\begin{itemize}
\item the set of admissible control values $U\subset\R^d$ is convex and compact;
\item the vector field $b$ is $C^{1,1}$ regular, i.e. {\bf Assumption (B)} in Section \ref{s-assB} below holds; 
\item the functions $f,\psi,g$ in $J$ are $C^{1,1}$ regular, i.e. {\bf Assumption (J)} in Section \ref{s-assJ} below holds;
\item  the function $\psi$ is $\lambda$-convex, 
i.e. {\bf Assumption (C)} in Section \ref{s-assJ} below holds.
\end{itemize}
Then, $(\mu^{\e},u^{\e})$ is the unique optimal pair for $(\mathcal{P}_{\e})$. Moreover, there exists $0<\Lambda'<\Lambda$, with $\Lambda$ provided by Theorem \ref{thm:main}, which depends only on the final time $T$ and the Lipschitz constant of the functions $b, f, g$, such that if $\lambda>\Lambda'$ then there exists a solution $(\mu,u)\in C([0,T];\Pn)\times L^\infty((0,T);\Lip(\R^d,U))$ of $\Pp$ and, up to a sub-sequence which we do not relabel, the following convergence reusults hold:
\begin{itemize}
\item[$(i)$] $u^{\e}\weakto u$ in  $L^2((0,T);W^{1,p}_\mathrm{loc}(\R^d,U))$ for every $1\leq p<\infty$;
\item[$(ii)$] $\mu^{\e}\to \mu$ in $C([0,T],\Pn)$;
\item[$(iii)$] $J(\mu,u)\leq \liminf_{n\to\infty}J(\mu^{\e},u^{\e})$.
\end{itemize}
\end{cor}

\begin{rem}
Our proof of existence of an optimal pair $(\mu,u)$ for of $\Pp$ and the convergence $(i)$ in the statements of the Theorem \ref{thm:main} and Corollary \ref{main-cor} are consequence of a compactness argument applied to the sequence $u^\e$, namely Lemma \ref{lem:comp-lip-maps} below. We cannot rule out the following phenomenon: one might have several minimizers for $\Pp$, and some among them might be the limits of sub-sequences of $u^\e$. 
\end{rem}

A central aspect of our proof is to show the uniform-in-$\e$ Lipschitz-continuity of the optimal control $u^\e$, with respect to the space variable. The proof exploits in a crucial way the relation between $u^\e$ and the adjoint solution of the mean-field forward-backward stochastic differential equation arising from (SOC). Typically, optimality conditions for minimizers can take the form of a system of partial differential equations coupling a Hamilton-Jacobi equation and the optimal control is given by a function of the derivative of the solution to the Hamilton-Jacobi equation. The difficulty is that, in general, the solution of a first order Hamilton-Jacobi equation is at most Lipschitz continuous: shocks of the gradient develop in finite time. Therefore, the optimal control can present discontinuities.

In order to analyze (SOC), we will use the techniques developed in \cite{Carmona2015}. It is worth to note that this kind of approach was also used in \cite{CCDelarue14}, where the authors prove the small time existence of smooth solutions of the master equation. In particular, they study the differentiability with respect to the initial condition of the flow generated by a forward-backward stochastic system of McKean-Vlasov type and they prove that the decoupling field generated by the forward-backward system is a classical solution of the corresponding master equation. The small time existence of a strong solution of the master equation, originating from the theory of Mean Field Games, has been proved also in \cite{Gangbo2015}, providing also a more rigorous connection between the master equation and the Mean Field Games equations. More recently, global in time solutions to the master equations for potential Mean Field Games are constructed in \cite{GangboMeszaros} for a class of Lagrangians and initial data which are displacement convex. Under sufficient regularity assumptions, solutions of mean-field control problems coincide with the equilibria of an associated the mean-field game. These equilibria are in turn given by solutions of the master equation. Then one might ask whether the results of \cite{GangboMeszaros} can directly guarantee the existence of solutions to the Problem $\Pp$. This is indeed a very delicate problem which is not yet covered in literature (to the best of our knowledge).

Finally, In the context of mean-field games, problems of a similar nature are addressed in \cite{Cardaliaguet2015}: the authors consider a class of mean-field games with local coupling and provide a stability analysis with respect to the diffusion term, which covers the case of vanishing viscosity. However, the system analyzed in \cite{Cardaliaguet2015} does not derive from a mean-field optimal control problem, see \cite{bensoussan1}, and therefore does not cover our result. We think that it is an interesting problem to understand whether it is possible to generalize the results of \cite{Cardaliaguet2015} for the systems derived in \cite{bensoussan1}, thus providing a different approach to the vanishing viscosity method.\\

The structure of the article is the following. In Section 2 we introduce some standard tools from analysis in the space of probability measures. Moreover, we recall existence and uniqueness results for the non-local continuity \eqref{eq:c} and advection-diffusion \eqref{eq:ad} equations. In Section 3 we define and solve a class of mean-field stochastic optimal control problems, which is closely related to our original problems $\Pp$ and $\Ppe$. Indeed, the results of Section 3 will provide the main building blocks for the proof of our main theorem, which will be given in Section 4. Finally, in Section 5 we will also show an example in which the vanishing viscosity limit does not hold if we drop the strict convexity assumption on the control cost $\psi$.

\section{Notations and preliminaries}
In this section, we fix notations and recall some notions of analysis in the space of probability measures, Wasserstein spaces, and non-local continuity equations.
\subsection{The Wasserstein distance}
We denote by $\mathcal{M}(\R^d)$ the set of measures on $\R^d$ and by $\Pb(\R^d)$ the subset of probability measures. The set of probability measures with compact support is denoted by $\Pc$, while $\Pac(\R^d)$ denotes the set of probability measures which are absolutely continuous with respect to the $d$-dimensional Lebesgue measure $\mathscr{L}^d$. We also define $\Pac_c(\R^d):=\Pac\cap \Pc$.

We say that a sequence $\{\mu^n\}_{n\in\N}\subset\Pb(\R^d)$  converges in the sense of measures towards $\mu\in\Pb(\R^d)$, denoted by $\mu^n\weaktos\mu$, provided that
\begin{equation}
    \lim_{n}\int_{\R^d}\phi(x)\mu^n(\de x)=\int_{\R^d}\phi(x)\mu(\de x),\,\hspace{0.4cm}\mbox{for all }\phi\in C^\infty_c(\R^d).
\end{equation}

The space $\Pb(\R^d)$ is equipped with the topology of the convergence of measures. For a given $p\geq 1$, we denote by $\Pb_p(\R^d)$ the set of probability measures with finite $p$-th moment $M_p$, which is defined as
\begin{equation}\label{def:pth-moment}
    M_p(\mu):=\int_{\R^d}|x|^p\mu(\de x).
\end{equation}
\begin{defn}
Let $\mu,\nu\in \Pb_p(\R^d)$. We say that $\gamma\in\mathscr{P}(\R^{2d})$ is a transport plan between $\mu$ and $\nu$ provided that $\gamma(A \times \R^d) = \mu(A)$ and $\gamma(\R^d \times B) = \nu(B)$ for any pair of Borel sets $A,B \subset \R^d$. We denote with $\Pi(\mu,\nu)$ the set of such transference plans.
\end{defn}
With these notations, we now introduce the Wasserstein distance in the space $\Pb_p(\R^d)$.
\begin{defn}
Given $p\geq 1$ and two measures $\mu,\nu\in \Pb_p(\R^d)$, the $p\,$-Wasserstein distance between $\mu$ and $\nu$ is 
\begin{equation}
W_p(\mu,\nu):=\inf_{\gamma\in \Pi(\mu,\nu)}\left\{ \int_{\R^d\times\R^d} |x-y|^p\,\gamma(\de x,\de y) \right\}^{1/p}.
\end{equation}
\end{defn}

We recall that the Wasserstein distance metrizes the weak-$*$ topology of probability measures; in particular the following holds, see \cite{villani,villani1}.

\begin{prop}\label{prop:conv_w}
The Wasserstein space $(\Pb_p(\R^d),W_p)$ is a complete and separable metric space. Moreover, for a given $\mu\in\Pb_p(\R^d)$ and a sequence of measures in $\mu^n\in\Pb_p(\R^d)$,  the following conditions are equivalent:
\begin{itemize}
    \item $W_p(\mu,\mu^n)\to 0$, as $n\to\infty$,
    \item $\mu^n\weaktos\mu$ and $\displaystyle\int_{\R^d}|x|^p\mu^n(\de x)\to \int_{\R^d}|x|^p\mu(\de x)$,
    \item $\mu^n\weaktos\mu$ and $\displaystyle\int_{B_R^c}|x|^p\mu^n(\de x)\to 0$ as $R\to \infty$ uniformly in $n$.
\end{itemize}
\end{prop}
We recall that Wasserstein distances are {\em ordered}, in the sense that, given $\mu,\nu\in\Pb_c(\R^d)$, then
\begin{equation}
\label{p-ordered}    
W_{p_1}(\mu,\nu)\leq W_{p_2}(\mu,\nu),\qquad \mbox{~~~whenever~~~}\qquad p_1\leq p_2.
\end{equation}

We also denote with $\Lip(\phi)$ a Lipschitz constant for a function $\phi$ and with $\Lip(X,Y)$ the space of Lipschitz functions from $X$ to $Y$, as well as $\Lip(X):=\Lip(X,\R)$. We now recall the {\em Kantorovich-Rubinstein} duality formula which characterizes the distance $W_1$, see \cite{villani1}.
\begin{lem}\label{lem:Kant-Rub}
Let $\mu,\nu\in\Pb_1(\R^d)$. Then
\begin{equation}
    W_1(\mu,\nu)=\sup_{\phi\in\Lip(\R^d)}\left\{\int_{\R^d}\phi(x) (\mu-\nu)(\de x):\,\,\Lip(\phi)\leq 1  \right\}.
\end{equation}
\end{lem}

\subsection{The L-derivative}
We now recall some results of differential calculus in the space of probability measures. Unless otherwise specified, all definitions and the results are taken from \cite{CarmonaDelarueI}. In particular, we choose a notion of derivative of a functional with respect to a measure, that suits our purposes. We recall that there are several different definitions of derivatives with respect to measures, see e.g. \cite{CarmonaDelarueI}. For our purpose, we need the so-called \emph{L-derivative}. Let $(\Omega,\mathcal{F},\mathbb{P})$ be an atomless probability space, where atomless means that for any $A\in \mathcal{F}$ with $\mathbb{P}(A)>0$ there exists $B\in \mathcal{F}$ such that $0< \mathbb{P}(B)< \mathbb{P}(A)$. 
\begin{defn}
Let $X:\Omega\to\R^d$ be a random variable. We define the {\em law} of $X$ the measure defined as $\Law(X)(B):=\mathbb{P}(X^{-1}(B))$, for any Borel set $B\subset\R^d$.
\end{defn}
The following proposition holds, see \cite[Proposition 9.1.11]{Bogachev}.
\begin{prop}\label{prop:w2-stoc}
Let $\mu\in \Pn$, then there exists a $\R^d$-valued random
variable $X\in L^2(\Omega;\R^d)$ with law $\Law(X)=\mu$. Moreover, if $\mu,\mu'\in\Pn$, then
$$
W_2(\mu,\mu')^2=\inf_{(X,X')}\E\left[|X-X'|^2 \right],
$$
where the infimum is taken over the pairs of $\R^d$-random variables $(X,X')$ such that $\Law(X)=\mu$ and $\Law(X')=\mu'$.
\end{prop}

Given a map $h:\Pn\to\R$ we define the {\em lift} $\tilde{h}:L^2(\Omega;\R^d)\to\R$ in the following way
$$
\tilde{h}(X)=h(\mathcal{L}(X)),\hspace{0.3cm}\forall X\in L^2(\Omega;\R^d).
$$
Note that $\mathcal{L}(X)\in\Pn$, since $X\in L^2(\Omega;\R^d)$. We point out that $L^2(\Omega,\mathcal{F},\mathbb{P})$ is an Hilbert space, in which the notion of Fr\'echet differentiability makes sense. We thus have the following definition.
\begin{defn}
A function $h:\Pn\to\R$ is said to be {\em L-differentiable} at $\mu_0\in\Pn$ if there exists a random variable $X_0$ with law $\mu_0$ such that the lifted function $\tilde{h}$ is Fr\'echet differentiable at $X_0$.
\end{defn}

The Fr\'echet derivative of $\tilde{h}$ at $X$ can be viewed as an element of $L^2(\Omega;\R^d)$; we denote it by $D\tilde{h}(X)$. It is important to recall that L-differentiability of $h$ does not depend upon the particular choice of $X$, as explained in the following propositions, see \cite{CarmonaDelarueI}.
\begin{prop}
Let $h:\Pn\to\R$ and $\tilde{h}$ its lift. Let $X,X'\in L^2(\Omega;\R^d)$ with the same law. If $\tilde{h}$ is Fr\'echet differentiable at $X$, then $\tilde{h}$ is Fr\'echet differentiable at $X'$ and $(X,D\tilde{h}(X))$ has the same law as $(X',D\tilde{h}(X'))$.
\end{prop}
\begin{prop}\label{prop:equivalence_class_xi}
Let $h:\Pn\to\R^d$ be an L-differentiable function. Then, for any $\mu_0\in\Pn$ there exists a measurable function $\xi:\R^d\to\R^d$ such that for all $X\in L^2(\Omega;\R^d)$ with law $\mu_0$, it holds that $D\tilde{h}(X)=\xi(X)$ $\mu_0$-almost surely.
\end{prop}
We say that $h$ is continuously L-differentiable if $D\tilde{h}$ is a continuous function from the space $L^2(\Omega,\mathcal{F},\mathbb{P})$ into itself. Moreover, by Proposition \ref{prop:equivalence_class_xi}, the equivalence class of $\xi\in L^2(\R^d,\mu_0;\R^d)$ is uniquely defined; we denote it by $\partial_\mu h(\mu_0)$. We call {\em L-derivative} of $h$ at $\mu_0$ the function
$$
\partial_\mu h(\mu_0)(\cdot):x\in\R^d\mapsto\partial_\mu h(\mu_0)(x).
$$
From the above construction, it is clear that $\partial_\mu h(\mu_0)(\cdot)$ is uniquely defined only $\mu_0$-a.e.. However, if $D\tilde{h}$ is a Lipschitz function from $L^2(\Omega,\mathcal{F},\mathbb{P})$ into itself, we can define a Lipschitz continuous version of $\partial_\mu h(\mu_0)(\cdot)$. This is the content of the following proposition, see \cite{CarmonaDelarueI}.

\begin{prop}\label{prop:lip_function_measure}
Assume that $(v(\mu)(\cdot))_{\mu\in\Pn}$ is a family of Borel-measurable mappings from $\R^d$ into itself for which there exists a constant $C>0$ such that, for any pair of identically distributed square integrable random variables $\xi_1,\xi_2\in L^2(\Omega,\mathcal{F},\mathbb{P};\R^d)$ over an atomless probability space $(\Omega,\mathcal{F},\mathbb{P})$, it holds:
\begin{equation*}
\E\left[ |v(\mathcal{L}(\xi_1))(\xi_1))-v(\mathcal{L}(\xi_2))(\xi_2)|^2 \right]\leq C^2\E\left[ |\xi_1-\xi_2|^2\right].
\end{equation*}
Then, for each $\mu\in\Pn$, one can redefine $v(\mu)(\cdot)$ on a $\mu$-negligible set in such a way that:
\begin{equation}
\forall x,x'\in\R^d,\qquad\mbox{~~~ it holds ~~~}\qquad |v(\mu)(x)-v(\mu)(x')|\leq C|x-x'|,
\end{equation}
for the same $C$ as above.
\end{prop}

To the above definition of differentiability we associate the following definition of convexity.
\begin{defn}\label{def:L-convex}
We say that a function $h:\Pn\to\R^d$ is {\em L-convex} if it is L-differentiable and satisfies
$$
h(\mu')\geq h(\mu)+\E[\partial_\mu h(\mu)(X)\cdot(X'-X)],
$$
whenever $X,X'\in L^2(\Omega;\R^d)$ have law $\mu,\mu'$, respectively.
\end{defn}

Finally, it is natural to extend the above definitions to functions depending on an $d$-dimensional variable $x$ and on a probability measure $\mu$, i.e. of the type $h:(x,\mu)\in\R^d\times\Pn\to\R$. With these notations, a function $h$ is {\em jointly differentiable} if its lift $\tilde{h}:\R^d\times L^2(\Omega;\R^d)$ is jointly differentiable. In particular, we can define partial derivatives $\partial_x h(x,\mu)$ and $\partial_\mu h(x,\mu)(x')$. We remark that joint continuous differentiability in the two arguments is equivalent to partial differentiability in each of the two arguments and joint continuity of the partial derivatives. Thus, the definitions and the results of this section can be easily generalized to this setting. In particular, if the derivatives of $h$ are Lipschitz, thanks to Proposition \ref{prop:lip_function_measure} we can find a Lipschitz continuous version of $\partial_\mu h(x,\mu)$ as a function 
$x'\in \R^d\mapsto \partial_\mu h(x,\mu)(x')$.

\subsection{Non-local continuity and diffusion equations} \label{s-assB}
We now provide a summary of the theory for the equations \eqref{eq:c} and \eqref{eq:ad}, based on \cite{MANITA2015, MANITA2014, Pedestrian}. We start by considering the Cauchy problem for the non-local continuity equation:
\begin{equation}\label{eq:cont}
    \begin{cases}
    \partial_t \mu_t+\dive[b(t,x,\mu_t)\,\mu_t]=0,\\
    \mu|_{t=0}=\mu_0,
    \end{cases}
\end{equation}
where the data of the problem are a fixed time horizon $T>0$, a vector field $b:(0,T)\times\R^d\times\Pb(\R^d)\to\R^d$ and the initial probability measure $\mu_0\in\Pn$. The above equation has to be understood in the sense of distributions, yielding to the following definition.
\begin{defn}\label{def:weak-sol-cont}
Let $\mu_0\in\Pn$. A weak solution of \eqref{eq:cont} is a probability measure $\mu\in C([0,T];\Pn)$ such that
\begin{equation*}
    \int_0^T\int_{\R^d}\left( \partial_t\varphi(t,x)+b(t,x,\mu_t)\cdot\nabla\varphi(t,x)\right)\mu_t(\de x)\de t=\int_{\R^d}\varphi(0,x)\mu_0(\de x),
\end{equation*}
for all test functions $\varphi\in C^\infty_c([0,T)\times\R^d)$.
\end{defn}
We remark that $\mu \in C([0,T];\Pn)$ means that the map $t\in [0,T]\mapsto \mu_t\in \Pn$ is continuous with respect to the weak convergence of measures, i.e. the map 
$$
t\in [0,T] \to \int_{\R^d}\varphi(x)\mu_t(\de x),
$$
is continuous for every $\varphi\in C^\infty_c([0,T)\times\R^d)$. Moreover, Definition \ref{def:weak-sol-cont} makes sense if 
$$
b(t,x,\mu_t)\in L^1((0,T);L^1_{\mathrm{loc}}(\R^d;\de \mu_t)).
$$
We will always work with vector fields satisfying the following assumptions.

\begin{center}
    \large {\bf Assumptions (B)}
\end{center}
\begin{enumerate}[label={\rm (B\arabic*)}]
\item \label{ass:vf1} The non-local velocity field $(t,x,\mu)\mapsto b(t,x,\mu)$ is measurable with respect to $t\in[0,T]$ and it is continuous in the $|\cdot|\times W_2$-topology with respect to $(x,\mu)\in\R^d\times\Pn$.
\item\label{ass:vf2} There exists $M>0$ such that
        \begin{equation}
        |b(t,x,\mu)|\leq M\left(1+|x|+\int_{\R^d}|y|\mu(\de y)\right),
    \end{equation}
    for all times $t\in[0,T]$ and any $(x,\mu)\in\R^d\times\Pn$. 
\item\label{ass:vf3} 
There exists a constant $L>0$ such that
        \begin{equation}
        |b(t,x,\mu)-b(t,y,\nu)|\leq L\left(|x-y|+W_2(\mu,\nu)\right),
        \end{equation}
    for all times $t\in[0,T]$ and for any $x,y\in \R^d$ and $\mu,\nu\in\Pn$.
\item\label{ass:vf4} The vector field $b$ is $L$-Lipschitz regular, i.e.
        \begin{eqnarray*}
        &&|\nabla_x b(t,x,\mu)-\nabla_x b(t,x',\mu')|\leq L\left(|x-x'|+W_2(\mu,\mu')\right),\\
        &&\E\left[ |\partial_\mu b(t,x',\mu')(X')-\partial_\mu b(t,x,\mu)(X)|^2 \right]\leq L^2\left( |x-x'|^2+\E\left[|X-X'|^2\right]\right),
        \end{eqnarray*}
        for all $t\in[0,T]$, $x,x'\in\R^d$ and $X,X'\in L^2(\Omega;\R^d)$ with law, respectively, $\mu,\mu'$.
\end{enumerate}

\begin{rem}
Both in {\bf Assumptions (B)}, and in the following {\bf Assumptions (J)} below, we denote by $M$ a constant related to boundedness, and by $L$ a constant related to Lipschitz continuity. In particular, the Lipschitz constant $L$ plays a crucial role in Theorem \ref{thm:main}, since we require $\lambda>\Lambda(T,L)$.
\end{rem}

Note that from \ref{ass:vf3} we also have that $$
|\nabla b(t,x,\mu)|+|\partial_\mu b(t,x,\mu)(x')|\leq L.
$$
By assuming the above hypotheses, the continuity equation \eqref{eq:c} admits a unique solution. We resume this well-posedness result in the following theorem, see \cite{Pedestrian}.
\begin{thm}\label{thm:well_pos_cont}
Let $b:(t,x,\mu)\in [0,T]\times\R^d\times\Pn\to\R^d$ be a vector field satisfying \ref{ass:vf1}, \ref{ass:vf2}, \ref{ass:vf3}. Then, for each $\mu_0\in\Pn$ there exists a unique solution $\mu\in C([0,T];\Pn)$ of \eqref{eq:cont}. Moreover, if $\supp\,\mu_0$ is compact, there exists a constant $r>0$ such that
$$
\supp\,\mu_t\subset B_{r}, \hspace{0.3cm}\mbox{for all }t,s\in[0,T].
$$
\end{thm}
 
We now study the diffusion equation.
\begin{equation}\label{eq:adv-diff}
\begin{cases}
\partial_t \mu_t^\e+\dive\left[b(t,x,\mu_t^\e)\mu_t^\e\right]=\e \Delta \mu_t^\e,\\
\mu^\e|_{t=0}=\mu_0.
\end{cases}
\end{equation}

First of all, a solution is a family of measures which satisfies the following.
\begin{defn}\label{def:weak-sol-ad}
Let $\mu_0\in\Pn$. A weak solution of \eqref{eq:adv-diff} is a probability measure $\mu^\e\in C([0,T];\Pn)$ such that
\begin{equation*}
    \int_0^T\int_{\R^d}\left( \partial_t\varphi(t,x)+b(t,x,\mu^\e_t)\cdot\nabla\varphi(t,x)+\e\Delta\varphi(t,x)\right)\mu_t^\e(\de x)\de t=\int_{\R^d}\varphi(0,x)\mu_0(\de x),
\end{equation*}
for all test functions $\varphi\in C^\infty_c([0,T)\times\R^d)$.
\end{defn}

By assuming the same regularity on the vector field, we have the following existence and uniqueness theorem for \eqref{eq:adv-diff}, see \cite{MANITA2015,MANITA2014}.
\begin{thm}
Let $b:(t,x,\mu)\in [0,T]\times\R^d\times\Pn\to\R^d$ be a vector field which satisfies \ref{ass:vf1}, \ref{ass:vf2}, \ref{ass:vf3}. Then, for each $\mu_0\in\Pn$ there exists a unique solution $\mu\in C([0,T];\Pn)$ of \eqref{eq:ad}. 
\end{thm}

We conclude this section with two technical lemmas.

\begin{lem}\label{lem:stimab}
Let $F$ be a scalar function or a vector field satisfying the regularity assumptions as in \eqref{ass:vf3} and \eqref{ass:vf4}. Then,
\begin{align}
    |F(t,x',\mu')-F(t,x,\mu)-\nabla F(t,x,\mu)\cdot(x'-x)&-\E\left[ \partial_\mu F(t,x,\mu)(X)\cdot(X'-X) \right]|\nonumber\\
    &\leq L\left(|x'-x|^2+W_2(\mu,\mu')^2\right),\label{est:abs_b}
\end{align}
for a.e. $t\in[0,T]$, for any $x,x'\in\R^d$, $\mu,\mu'\in \Pn$ and $X_t,X_t'\in L^2((0,T);L^2(\Omega;\R^d))$ with law $\mu,\mu'$ respectively.
In particular, we have that
\begin{align}
    \E &\left[\int_0^T|F(t,X_t',\Law(X_t'))-F(t,X_t,\Law(X_t))-\nabla F(t,X_t,\Law(X_t))\cdot(X_t'-X_t)\right.\nonumber\\
    &\left.-\tilde{\E}\left[ \partial_\mu F(t,X_t,\Law(X_t))(\tilde{X}_t)\cdot(\tilde{X}_t'-\tilde{X}_t) \right]\right|\de t\biggr]\leq 2L\E\left[\int_0^T|X_t'-X_t|^2\de t\right],\label{est:square_b}
\end{align}
for any square integrable process $X_t,X_t'\in L^2((0,T);L^2(\Omega;\R^d))$.
\end{lem}
\begin{proof}
We assume that $F$ is a scalar, the proof for the vectorial case follows in the same way.
We add and subtract the quantity $F(t,x,\mu')$ in the absolute value on the left hand side. From the identity
\begin{align*}
F(t,x',\mu')-F(t,x,\mu')=\int_0^1\nabla F(t,sx'+(1-s)x,\mu')\cdot(x'-x)\de s,
\end{align*}
it easily follows that
\begin{align*}
|F(t,x',\mu')&-F(t,x,\mu')-\nabla F(t,x,\mu)\cdot(x'-x)|\\
& =\left|\int_0^1\nabla F(t,sx'+(1-s)x,\mu')\cdot(x'-x)\de s-\nabla F(t,x,\mu)\cdot(x'-x)\right|\\
& \leq \int_0^1|\nabla F(t,sx'+(1-s)x,\mu')-\nabla F(t,x,\mu)||x'-x|\de s\\
& \leq L |x'-x|^2\int_0^1 s \de s +\frac{L}{2}|x'-x|^2+\frac{L}{2}W_2(\mu,\mu')^2\\
& \leq L|x'-x|^2+\frac{L}{2}W_2(\mu,\mu')^2,
\end{align*}
where in the third line we used Young's inequality. On the other hand, from the identity
\begin{align*}
F(t,x,\mu')-F(&t,x,\mu)=\E\left[\int_0^1\partial_\mu F(t,x,\Law(sX'+(1-s)X)(sX'+(1-s)X)\cdot(X'-X)\de s\right],
\end{align*}
it follows that
\begin{align*}
\E &\Bigg[\int_0^1\Bigg(\partial_\mu F(t,x,\Law(sX'+(1-s)X)(sX'+(1-s)X)-\partial_\mu F(t,x,\Law(X))(X)\Bigg)\cdot(X'-X)\de s\Bigg]\\
&\leq \E \int_0^1\big|\partial_\mu F(t,x,\Law(sX'+(1-s)X)(sX'+(1-s)X) -\partial_\mu F(t,x,\Law(X))(X)\big| |X'-X|\de s\\
& \leq \frac{L}{2} \E[|X'-X|^2].
\end{align*}
Then, the conclusion follows from the triangle inequality, the two estimates above, and Proposition \ref{prop:w2-stoc}. Finally, \eqref{est:square_b} follows from \eqref{est:abs_b}.
\end{proof}

We now recall the Osgood's lemma, see \cite{Chemin96, CCS21}.
\begin{lem}\label{lem:osgood}
Let $\rho$ be a positive Borel function, $\gamma$ a locally integrable positive function, $\psi$ a continuous increasing strictly positive function, and $\eta>0$. Assume that the function $\rho$ satisfies one between
\begin{equation}\label{inequality-osgood}
    \rho(t)\leq \eta+\int_{t_0}^t\gamma(s)\psi(\rho(s))\de s, \hspace{0.5cm}\mbox{or}\hspace{0.5cm}\rho(t)\leq \eta+\int_t^{t_0}\gamma(s)\psi(\rho(s))\de s.
\end{equation}
Define the function $\mathfrak{M}$ as
$$
\mathfrak{M}(x)=\int_x^1\frac{1}{\psi(s)}\de s.
$$
Then, in the first case it holds that
$$
-\mathfrak{M}(\rho(t))+\mathfrak{M}(\eta)\leq \int_{t_0}^t\gamma(s)\de s,
$$
while in the second case
$$
-\mathfrak{M}(\rho(t))+\mathfrak{M}(\eta)\leq \int_t^{t_0}\gamma(s)\de s.
$$
\end{lem}
\begin{proof}
A proof of the lemma when $\rho$ satisfies the first inequality in \eqref{inequality-osgood} can be found in \cite{Chemin96}. We prove the lemma in the case $\rho$ satisfies the second inequality in \eqref{inequality-osgood}, which is stated but not proved in \cite{CCS21}. Define the function $\displaystyle R_\eta(t):= \eta+\int_t^{t_0}\gamma(s)\psi(\rho(s))\de s$. That implies $R_\eta(t)\geq \rho(t)$ by assumption. Since $R$ is absolutely continuous, then it holds
$$
\dot{R}_\eta(t)=-\gamma(t)\psi(\rho(t))\geq -\gamma(t)\psi(R_\eta(t)) \hspace{0.3cm}\mbox{for a.e. }t,
$$
and integrating the above expression in time
$$
\int_t^{t_0}\frac{\dot{R}_\eta(s)}{\psi(R_\eta(s))}\de s\geq -\int_t^{t_0}\gamma(s)\de s.
$$
Then, by the change of variables $s\to R_\eta(s)$ in the left hand side, we get that
$$
\int_{R_\eta(t)}^\eta \frac{\de s}{\psi(s)}\geq -\int_t^{t_0}\gamma(s)\de s.
$$
By the definition of $\mathfrak{M}$, the left hand side coincides with $\mathfrak{M}(\eta)-\mathfrak{M}(R_\eta(t))$.
Then, by using that $\mathfrak{M}$ is decreasing we obtain that
$$
-\mathfrak{M}(\rho(t))+\mathfrak{M}(\eta)\leq -\mathfrak{M}(R_\eta(t))+\mathfrak{M}(\eta)\leq \int_t^{t_0}\gamma(s)\de s
$$
and this concludes the proof.
\end{proof}

\section{Stochastic mean-field optimal control}
In this section we provide a short overview of stochastic control theory. After introducing the notations, we will give an appropriate version of the Pontryagin Maximum Principle  based on the tools introduced in Section 2. Most of the results are taken from \cite{Carmona2015, CarmonaDelarueI} and slightly modified to fit our context.

\subsection{The stochastic set-up} \label{s-assJ}
Let $(\Omega, \mathcal{F}_t,\mathbb{P})$ be a complete filtered probability space, equipped with an adapted Brownian motion $W_t$. Without any loss of generality, we can assume that $(\Omega, \mathcal{F}_0,\mathbb{P})$ is an atomless probability space and the filtration $\mathcal{F}_t$ is the one generated by $\mathcal{F}_0$ and $W_t$.

We denote by $X_t^\e$ a stochastic process solving the following stochastic differential equation:
\begin{equation}
\label{eq:SDE-generica}
\begin{cases}
\de X^\e_t= \left(b(t,X^\e_t,\mathcal{L}(X^\e_t))+\alpha_t \right) \de t +\sqrt{2\e}\,\de W_t,\\
X_0^\e=\xi,
\end{cases}
\end{equation}
where $\xi\in L^2(\Omega,\mathcal{F}_0,\mathbb{P};\R^d)$ is a given random variable with $\Law(\xi)=\mu_0$, and the admissible control $\alpha_t$ satisfies:
\begin{center}
    \large\textbf{Assumption (A)}
\end{center}
\begin{enumerate}[label={\rm (A\arabic*)}]
    \item \label{ass:setofcontrols} $\alpha_t$ is a measurable process with values in $U$. Define $R:=\max(1,\max_{x\in U}|x|)$.
\end{enumerate}
\vspace{0.5cm}
In particular, since $U$ is compact, the following bound is trivial
\begin{equation}\label{admiss-contr}
    \E \displaystyle\int_0^T |\alpha_t|^2\de t<\infty.
\end{equation}

We now state a classical well-posedness result for \eqref{eq:SDE-generica}, see \cite{CarmonaDelarueI}.
\begin{thm}\label{thm:well-posedness-fsde}
If $b$ satisfies {\bf Assumptions (B)} and $\alpha_t$ satisfies \eqref{admiss-contr}, then there exists a unique solution $X^\e_t$ of \eqref{eq:SDE-generica} such that
\begin{equation*}
    \E\left[ \sup_{t\in (0,T)}|X^\e_t|^2 \right]<\infty.
\end{equation*}
Moreover, if $X^{\e,'}_t$ is a solution of \eqref{eq:SDE-generica} with control $\alpha_t'$ and initial condition $\xi'$, then the following stability estimate holds
\begin{equation}\label{stabilityFSDE}
\E\left[ \sup_{t\in(0,T)}|X_t^{\e,'}-X_t^\e|^2 \right]\leq C_2(T,L)\left(\E\left[ |\xi'-\xi|^2 \right]+\E\left[ \int_0^T|\alpha_t'-\alpha_t|^2\de t \right]\right),
\end{equation}
where $L$ is the Lipschitz constant of $b$ in \ref{ass:vf3}.
\end{thm}

It is worth to remark that in Theorem \ref{thm:well-posedness-fsde} the existence is understood in the {\em strong} sense, i.e.  one can find a solution to \eqref{eq:SDE-generica} on any given filtered probability space equipped with any given adapted Brownian
motion. Moreover, {\em pathwise uniqueness} holds: it means that, on any given filtered probability space equipped with any given Brownian motion, any two solutions to \eqref{eq:SDE-generica} with the same initial condition $\xi$ coincide. Since it will be crucial for the following, we compute the constants appearing in Theorem \ref{thm:well-posedness-fsde} under our specific assumptions. We recall that $M_2(\mu_0)$ is the second moment defined in \eqref{def:pth-moment}.

\begin{lem}\label{lem:constant-fsde}
Under the hypothesis of Theorem \ref{thm:well-posedness-fsde}, for $\e\leq 1$, we have that
\begin{equation*}
     \E\left[ \sup_{t\in (0,T)}|X^\e_t|^2 \right]\leq C_1(T,\mu_0,M,R),
\end{equation*}
with
\begin{equation}\label{def:C_1}
    C_1(T,\mu_0,M,R):=\left(\frac{M+R}{M+1}+\sqrt{M_2(\mu_0)+T} \right)^2e^{(M+1)T},
\end{equation}
where $M$ is the constant in \ref{ass:vf2} and $R$ as in \ref{ass:setofcontrols}. Moreover, the constant $C_2$ in \eqref{stabilityFSDE} is
\begin{equation}\label{def:C_2}
    C_2(T,L):=\exp\{(4L+1)T\}.
\end{equation}
Moreover, we also have the estimate
\begin{equation}\label{eq:further-estimate}
    \E\left[\sup_{t\in(0,T)}|X_t^{\e,'}-X_t^\e|^2\right]\leq \tilde C_2(T,L)\left(\E\left[|\xi'-\xi|^2\right]+T\int_0^T\E\left[|\alpha_t'-\alpha_s|^2\right]\de s\right),
\end{equation}
where $\tilde C(T,L)=Ce^{CL^2T}$, for some positive constant $C$.
\end{lem}
\begin{proof}
We divide the proof in two steps.\\
\\
\underline{\em Step 1}  \hspace{0.3cm} {\bf Uniform bound on the forward component.}
Let $X_t^\e$ be the solution of \eqref{eq:SDE-generica} with control $\alpha_t$ and initial datum $\xi_0$. By an easy application of Ito's Lemma, we get that
\begin{equation*}
    \E[|X_t^\e|^2]=\E[|\xi|^2]+\E\left[\int_0^T (b(t,X_t^\e,\Law(X_t^\e))+\alpha_t)\cdot X_t^\e\de t\right]+\e t.
\end{equation*}
Notice that $|\alpha_t|\leq R$ for all $t\in[0,T]$, since it takes values in the compact set $U$. Then, by using the growth assumptions \ref{ass:vf2} on $b$, $L^\infty$ bound on $\alpha_t$, Lemma \ref{lem:osgood}, and Doob's inequality \cite{kunita} we get
\begin{equation}
    \E\left[\sup_{t\in(0,T)}|X_t^\e|^2\right]\leq \left(\frac{M+R}{M+1}+\sqrt{M_2(\mu_0)+T} \right)^2e^{(M+1)T}.
\end{equation}
\\
\underline{ \em Step 2} \hspace{0.3cm} {\bf Stability estimate for the forward component.}
Let $X_t^\e,X_t^{\e,'}$ be, respectively, the solutions of \eqref{eq:SDE-generica} with control $\alpha_t,\alpha_t'$ and initial condition $\xi,\xi'$. Then, the difference $X_t^{\e,'}-X_t^\e$ satisfies the ordinary differential equation
\begin{equation}\label{eq:differenza-forward}
\begin{cases}
    \frac{\de}{\de t} (X_t^{\e,'}-X_t^\e) = b(t,X_t^{\e,'},\Law(X_t^{\e,'}))+\alpha_t'-b(t,X_t^\e,\Law(X_t^\e))-\alpha_t,\\
    X_0^{\e,'}-X_0^\e=\xi'-\xi,
\end{cases}
\end{equation}
and by using the Lipschitz property \ref{ass:vf3} and standard estimates we get that
\begin{equation}
    \E\left[\sup_{t\in(0,T)}|X_t^{\e,'}-X_t^\e|^2\right]\leq \exp\{(4L+1)T\}\left(\E\left[|\xi'-\xi|^2+\int_0^T|\alpha_t'-\alpha_t|^2\de t\right]\right).
\end{equation}
\underline{ \em Step 3} \hspace{0.3cm} {\bf Further estimate on the stability constant.}
We now provide a different estimate on the difference. From \eqref{eq:differenza-forward} and the assumption \ref{ass:vf3} we obtain that
\begin{align*}
    |X_t^{\e,'}-X_t^\e|&\leq|\xi'-\xi|+ L\int_0^t|X_s^{\e,'}-X_s^\e|\de s+L\int_0^t\sqrt{\E|X_s^{\e,'}-X_s^\e|^2}\de s\\&+\int_0^t|\alpha_s'-\alpha_s|\de s.
\end{align*}
Taking the square of both sides and using the Cauchy-Schwarz inequality we get
\begin{align*}
    |X_t^{\e,'}-X_t^\e|^2 &\leq C|\xi'-\xi|^2+CL^2\left(\int_0^t|X_s^{\e,'}-X_s^\e|\de s\right)^2\\
    &+CL^2\left(\int_0^t\sqrt{\E|X_s^{\e,'}-X_s^\e|^2}\de s\right)^2+C\left(\int_0^t|\alpha_s'-\alpha_s|\de s\right)^2\\
    &\leq C|\xi'-\xi|^2+ CL^2 t\int_0^t|X_s^{\e,'}-X_s^\e|^2\de s\\
    &+CL^2t\int_0^t\sqrt{\E|X_s^{\e,'}-X_s^\e|^2}\de s+Ct\int_0^t|\alpha_s'-\alpha_s|^2\de s.
\end{align*}
Then, taking the expectation, with standard arguments we obtain that
\begin{equation}
    \E\left[\sup_{t\in(0,T)}|X_t^{\e,'}-X_t^\e|^2\right]\leq Ce^{CL^2T}\left(\E\left[|\xi'-\xi|^2\right]+T\int_0^T\E\left[|\alpha_s'-\alpha_s|^2\right]\de s\right),
\end{equation}
and this concludes the proof.
\end{proof}

We now proceed by defining the following stochastic optimal control problem.

\begin{pschema}{Stochastic optimal control problem (SOC)}
Denote by $\bm{\alpha}:=(\alpha_t)_{0\leq t\leq T}$ the control on the whole time interval, and consider the cost functional
\begin{equation}
\label{def:costo_stocastico}
J^S(\bm{\alpha})=\E\left[ g(X^\e_T,\mathcal{L}(X^\e_T))+\int_0^T \left(f(t,X^\e_t,\mathcal{L}(X^\e_t))+\psi(\alpha_t)\right)\de t \right].
\end{equation}
Find
\begin{equation*}
    \min_{\bm{\alpha}}J^S(\bm{\alpha}),
\end{equation*}
such that $X_t^\e$ is a solution of \eqref{eq:SDE-generica} and the control satisfies {\bf Assumption (A)}.
\end{pschema}

From now on, we assume the following hypothesis on the cost $J$.
\begin{center}
    \large {\bf Assumptions (J)}
\end{center}
\begin{enumerate}[label={\rm (J\arabic*)}]
    \item \label{ass:controlcost} The control cost $\alpha\in U\to \psi(\alpha)\in\R$ is $C^{1}$.
     \item\label{ass:cost} The functions $f$ and $g$ are $C^{1}$ and $L$-Lipschitz: for all $t\in[0,T]$, $x,x'\in\R^d $, $\mu,\mu'\in \Pn$ it holds
        \begin{eqnarray*}
        &&|f(t,x,\mu)-f(t,x',\mu')|\leq 
        L\left(|x-x'|+W_2(\mu,\mu')\right),\\
        &&|g(x,\mu)-g(x',\mu')|\leq 
        L\left(|x-x'|+W_2(\mu,\mu')\right).
        \end{eqnarray*}
        \item\label{ass:regularitycost} The derivatives of $f$ and $g$ with respect to $x$ are $L$-Lipschitz continuous with respect to $(x,\mu)$, i.e. for every $t\in[0,T]$, $x,x'\in\R^d$, and $\mu,\mu'\in\Pn$ it holds
        \begin{eqnarray*}
        &&|\nabla_x f(t,x,\mu)-\nabla_x f(t,x',\mu')|\leq L\left(|x-x'|+W_2(\mu,\mu')\right),\\
        &&|\nabla_x g(x,\mu)-\nabla_x g(x',\mu')|\leq L\left(|x-x'|+W_2(\mu,\mu')\right),
        \end{eqnarray*}
         and for all $X,X'\in L^2(\Omega;\R^d)$ with law, respectively, $\mu,\mu'$ it holds
        \begin{align*}
        \E\left[ |\partial_\mu f(t,x',\mu')(X')-\partial_\mu f(t,x,\mu)(X)|^2 \right]\leq L^2\left( |x-x'|^2+\E\left[|X-X'|^2\right]\right),
        \end{align*}
        \begin{align*}
        \E\left[ |\partial_\mu g(x',\mu')(X')-\partial_\mu g(x,\mu)(X)|^2 \right]\leq L^2\left( |x-x'|^2+\E\left[|X-X'|^2\right]\right).
        \end{align*}
\end{enumerate}

Note that from {\color{red}\ref{ass:cost}} we also have that the derivatives of $f$ and $g$ are bounded, i.e for every $t\in[0,T]$, $x,x'\in\R^d$, and $\mu\in\Pn$ it holds        \begin{equation}\label{ass:derivate_fg}
    |\nabla_x f(t,x,\mu)|,\,\,|\nabla_x g(x,\mu)|\leq L,\hspace{0.6cm}
    |\partial_\mu f(t,x,\mu)(x')|,\,\,|\partial_\mu g(x,\mu)(x')|\leq L.
\end{equation}
Moreover, under the assumptions \ref{ass:cost} and \ref{ass:regularitycost}, we can apply Lemma \ref{lem:stimab} to $f$ and $g$.
We associate to (SOC) the Hamiltonian:
\begin{equation}
\label{def:hamiltoniana_stocastico}
H(t,x,\mu,y,\alpha)=(b(t,x,\mu)+\alpha)\cdot y+f(t,x,\mu)+\psi(\alpha),
\end{equation}
for $(t,x,\mu,y,\alpha)\in [0,T]\times\R^d\times\Pn\times\R^d\times U$. It is worth to note that under {\bf Assumptions (A), (B), (J)}, the Hamiltonian $H$ is $C^{1,1}_\mathrm{loc}$-regular.

In full analogy with the Pontryagin Maximum Principle for finite-dimensional control problems, we introduce an adjoint process as the solution of a backward equation involving partial derivatives of the Hamiltonian with respect to the measure argument. Then, for a given admissible control $\alpha_t$ and the corresponding controlled state $X^\e_t$, we give the following definition.
\begin{defn}
We call {\em adjoint processes} of $X_t^\e$ any couple $(Y_t^\e,Z_t^\e)$ satisfying the backward stochastic equation
\begin{equation}\label{eq:adj_st-generica}
    \begin{cases}
    \de Y_t^\e=-\left[\nabla_x H(t,X^\e_t,\Law(X^\e_t),Y^\e_t,\alpha_t)+\tilde{\E}[\partial_\mu H(t,\tilde{X}^\e_t,\Law(\tilde{X}^\e_t),\tilde{Y}^\e_t,\tilde{\alpha}_t)(X^\e_t)]\right]\de t+Z^\e_t\de W_t,\\
    Y^\e_T=\nabla g(X^\e_T,\Law(X^\e_T))+ \tilde{\E}\left[\partial_\mu g(\tilde{X}^\e_T,\Law(\tilde{X}^\e_T))(X^\e_T)\right],
    \end{cases}
\end{equation}
where $(\tilde{X}^\e_t,\tilde{Y}^\e_t,\tilde{\alpha}_t)$ is an independent copy of $(X^\e_t,Y^\e_t,\alpha_t)$ defined on the space $(\tilde{\Omega},\tilde{\mathcal{F}}_t,\tilde{\mathbb{P}})$ and $\tilde{\E}$ denotes the expectation on $(\tilde{\Omega},\tilde{\mathcal{F}}_t,\tilde{\mathbb{P}})$. In particular, it holds $\Law(\tilde{X}^\e_t)=\Law(X^\e_t)$.
\end{defn}
A solution of \eqref{eq:adj_st-generica} is a couple $(Y^\e_t,Z^\e_t)$, where the introduction of the process $Z^\e_t$ is necessary to ensure that the process $Y^\e_t$ is adapted with respect to the {\em forward filtration} $\mathcal{F}_t$.

It is worth to note that the functions $\nabla_x H$ and $\partial_\mu H$ do not depend on $\alpha$, due to the particular form of the Hamiltonian \eqref{def:hamiltoniana_stocastico}. As a consequence, the system can be rewritten as
\begin{equation}\label{eq:adj_st-specifica}
    \begin{cases}
    \de Y_t^\e &=-\left[\nabla_x b(t,X^\e_t,\Law(X^\e_t))Y^\e_t+\nabla_x f(t,X^\e_t,\Law(X^\e_t))\right]\de t\\
    & -\left[\tilde{\E}[\partial_\mu b(t,\tilde{X}^\e_t,\Law(\tilde{X}^\e_t))(X^\e_t)\tilde{Y}^\e_t+\partial_\mu f(t,\tilde{X}^\e_t,\Law(\tilde{X}^\e_t)(X^\e_t)]\right]\de t+Z^\e_t\de W_t,\\
    Y^\e_T &=\nabla g(X^\e_T,\Law(X^\e_T))+ \tilde{\E}\left[\partial_\mu g(\tilde{X}^\e_T,\Law(\tilde{X}^\e_T))(X^\e_T)\right],
    \end{cases}
\end{equation}

The equation \eqref{eq:adj_st-generica} is a {\em backward stochastic differential equation} (BSDE) of mean-field type, since the law of $Y^\e_t$ appears in the term which involves the L-derivative of $H$. This kind of BSDE admits a unique solution, if we assume enough regularity on the coefficients and we consider $X^\e_t,\alpha_t$ as given data of the problem. In particular, the following theorem holds, see \cite{CarmonaDelarueI}.

\begin{thm}\label{thm:well-posedness-bsde}
Let $\alpha_t$ be an admissible control and $X^\e_t$ the corresponding trajectory. Under {\bf Assumptions (A), (B), (J)}, there exists a unique solution $(Y^\e_t,Z^\e_t)$ such that
\begin{equation}\label{boundYZ}
    \E\left[ \sup_{t\in(0,T)}|Y^\e_t|^2+\int_0^T|Z^\e_t|^2\de t \right]<\infty.
\end{equation}
Moreover, if $(Y^{\e,'}_t,Z^{\e,'}_t)$ is a solution corresponding to a control $\alpha_t'$ and a stochastic process $X^{\e,'}_t$, it holds that
\begin{equation}\label{stability_bsde}
    \E\left[ \sup_{t\in(0,T)}|Y^{\e,'}_t-Y^\e_t|^2+\int_0^T|Z^{\e,'}_t-Z^\e_t|^2\de t \right]\leq C(T,L)\E\left[ \sup_{t\in(0,T)}|X_t^{\e,'}-X_t^\e|^2 \right],
\end{equation}
where $L$ is the Lipschitz constant appearing in {\bf Assumptions (B), (J)}.
\end{thm}

Similarly to what we have done for the forward component, we explicitly compute the constants ensuring boundedness of $Y^\e_t$ and well-posedness of solutions of \eqref{eq:adj_st-specifica}.
\begin{lem}\label{lem:constant-bsde}
Let $\alpha_t$ be an admissible control and $X_t^\e$ the corresponding trajectory. Under the assumptions of Theorem \ref{thm:well-posedness-bsde}, for every $p\geq 1$ it holds
\begin{equation}
    \E\left[ \sup_{t\in(0,T)}|Y_t^\e|^p \right]\leq C_3(T,L)^p,
\end{equation}
where 
\begin{equation}\label{def:C_3}
    C_3(T,L):=(1+2L)e^{2LT}.
\end{equation}
Finally, given another trajectory $X^{\e,'}_t$, we have that 
\begin{equation}\label{lem:stima-backward}
\E\left[ \sup_{t\in(0,T)}|Y^{\e,'}_t-Y^\e_t|^2\right]\leq C_4(T,L)\E\left[\sup_{t\in(0,T)}|X_t^{\e,'}-X_t^\e|^2\right]
\end{equation}
and the constant $C_4(T,L)$ is given by
\begin{equation}\label{def:C_4}
    C_4(T,L):=4L^2(2+T+TC_3(T,L)^2) e^{(6+2L^2)T}.
\end{equation}
\end{lem}
\begin{proof}
We divide the proof in two steps.\\
\\
\underline{\em Step 1} \hspace{0.3cm} {\bf Uniform bounds on $Y^\e_t$.}
Let $p\geq 2$, an application of Ito's lemma gives that
\begin{align*}
    \E\left[|Y_t^\e|^p\right]+& p(p-1)\E\left[\int_t^T|Z_s^\e|^2|Y^\e_s|^{p-2}\de s\right] = \E\left[ |Y^\e_T|^p\right]\\
    &+ p\E\left[ \int_t^T \left(\nabla_x b(s,X^\e_s,\Law(X^\e_s))Y^\e_s+\nabla_x f(s,X^\e_s,\Law(X^\e_s))\right) \cdot Y^\e_s|Y^\e_s|^{p-2}\de s \right]\\
    &+p\E\left[ \int_t^T \left(\tilde{\E}[\partial_\mu b(s,\tilde{X}^\e_s,\Law(\tilde{X}^\e_s))(X^\e_s)\tilde{Y}^\e_s+\partial_\mu f(s,\tilde{X}^\e_s,\Law(\tilde{X}^\e_s)(X^\e_s)] \right) \cdot Y^\e_s|Y^\e_s|^{p-2}\de s \right].
\end{align*}
We estimate the terms on the right hand side separately. First, by \eqref{ass:derivate_fg} we have that
\begin{equation}
    \E\left[ |Y^\e_T|^p\right]\leq (2L)^p.
\end{equation}
We now consider the term involving $\nabla_x b$: by using assumption \ref{ass:vf3} we easily obtain that
\begin{equation}
    \E\left[ \int_t^T \nabla_x b(s,X^\e_s,\Law(X^\e_s))Y^\e_s\cdot Y^\e_s|Y^\e_s|^{p-2}\de s \right]\leq L \int_t^T\E\left[|Y^\e_s|^p\de s\right].
\end{equation}
On the other hand, for the part involving $\partial_\mu b$ we use assumption \ref{ass:vf4} and Holder's inequality, obtaining
\begin{align}
    \E\Bigg[ \int_t^T \tilde{\E}\big[\partial_\mu b(s,\tilde{X}^\e_s,\Law(\tilde{X}^\e_s))(X^\e_s)\tilde{Y}^\e_s\big]&\cdot Y^\e_s|Y^\e_s|^{p-2}\de s \Bigg] \leq L\int_t^T\tilde{\E}[|Y_s^\e|]\E[|Y^\e_s|^{p-1}]\de s \nonumber\\
    & \leq L\int_t^T \E[|Y_s^\e|^p]^{\frac{1}{p}}\E[|Y^\e_s|^p]^{\frac{p-1}{p}}\de s=L\int_t^T \E[|Y_s^\e|^p]\de s.
\end{align}
Finally, for the terms involving $f$ we use \eqref{ass:derivate_fg} to get
\begin{equation}
    \E\left[\int_t^T \nabla_x f(s,X^\e_s,\Law(X^\e_s))\cdot Y^\e_s|Y^\e_s|^{p-2}\de s \right]\leq L \int_t^T\E\left[|Y^\e_s|^p\right]^{\frac{p-1}{p}}\de s,
\end{equation}
\begin{equation}
    \E\left[\int_t^T\tilde{\E}\left[\partial_\mu f(s,\tilde{X}^\e_s,\Law(\tilde{X}^\e_s)(X^\e_s)\right]\cdot Y^\e_s|Y^\e_s|^{p-2}\de s \right]\leq L \int_t^T\E\left[|Y^\e_s|^p\right]^{\frac{p-1}{p}}\de s.
\end{equation}
Putting together the previous estimates, we obtain
\begin{equation}
    \E\left[ |Y^\e_t|^p\right]\leq (2L)^p+2Lp \int_t^T\E[|Y^\e_s|^p]+\E\left[|Y^\e_s|^p\right]^{\frac{p-1}{p}}\de s.
\end{equation}
By defining $y(t):=\E\left[ |Y^\e_t|^p\right]$, the above inequality can be rewritten as
\begin{equation}
    y(t)\leq (2L)^p+2Lp \int_t^T y(s)+y(s)^{\frac{p-1}{p}}\de s,
\end{equation}
and an application of Lemma \ref{lem:osgood} provides the following estimate
\begin{equation}
    y(t)^{1/p}\leq (1+2L)e^{2LT}.
\end{equation}
In other words, we get that
\begin{equation}\label{bound_lp}
    \|Y_t^\e\|_{L^p(\Omega)}\leq (1+2L)e^{2LT}.
\end{equation}
The same bound trivially holds for $1\leq p<2$ since we are working on a probability space. Moreover, it is a classical fact that, since the right hand side of \eqref{bound_lp} is uniformly bounded in $p$, then $Y^\e_t\in L^\infty(\Omega)$ (see \cite[Exercise 1.3.5]{Tao-book}) and moreover
\begin{equation}\label{est:uniform_bound_y}
    \sup_{t\in(0,T)}\sup_{\omega\in\Omega}|Y_t^\e|\leq (1+2L)e^{2LT}.
\end{equation}

\underline{\em Step 2} \hspace{0.3cm} {\bf Stability estimate on the backward component $Y_t^\e$.}
Let $X_t^\e,X_t^{\e,'}$ be, respectively, the solutions of \eqref{eq:SDE-generica} with control $\alpha_t,\alpha_t'$ and initial condition $\xi,\xi'$. We denote with $Y^\e_t,Y_t^{\e,'}$ the associate adjoint processes. Applying Ito's lemma to $|Y_t^{\e,'}-Y_t^\e|^2$ we get that
\begin{align*}
    \E\left[ |Y_t^{\e,'}-Y_t^\e|^2 \right]&+\E\left[\int_t^T|Z_s^{\e,'}-Z_s^\e|^2\de s\right]= \E\left[ |Y_T^{\e,'}-Y_T^\e|^2 \right]\\
    +2\E\left[ \int_t^T \biggl( \right.&\left.\left.\nabla_x b(s,X^{\e,'}_s,\Law(X^{\e,'}_s))Y^{\e,'}_s-\nabla_x b(t,X^\e_s,\Law(X^\e_s)) Y^\e_s \right)\cdot(Y_s^{\e,'}-Y^\e_s) \de s\right]\\
    +2\E\left[ \int_t^T \biggl( \right.&\left.\left.\nabla_x f(s,X^{\e,'}_s,\Law(X^{\e,'}_s))-\nabla_x f(s,X^\e_s,\Law(X^\e_s)) \right)\cdot(Y_s^{\e,'}-Y^\e_s) \de s \right]\\
    +2\E\left[ \int_t^T \biggl( \right.&\left.\left.\tilde{\E}\left[ \partial_\mu b(s,\tilde{X}^{\e,'}_s,\Law(\tilde{X}^{\e,'}_s))(X^{\e,'}_s)\tilde{Y}^{\e,'}_s-\partial_\mu b(s,\tilde{X}^\e_s,\Law(\tilde{X}^\e_s))(X^\e_s)\tilde{Y}^\e_s  \right]\right)\cdot(Y_s^{\e,'}-Y^\e_s) \de s\right]\\
    +2\E\left[ \int_t^T \biggl( \right.&\left.\left.\tilde{\E}\left[ \partial_\mu f(s,\tilde{X}^{\e,'}_s,\Law(\tilde{X}^{\e,'}_s))(X^{\e,'}_s)-\partial_\mu f(s,\tilde{X}^\e_s,\Law(\tilde{X}^\e_s))(X^\e_s)  \right]\right)\cdot(Y_s^{\e,'}-Y^\e_s) \de s\right].
\end{align*}
We estimate the terms in the above inequality separately. First, by using \ref{ass:regularitycost}, for the part involving the final datum we have that
\begin{align*}
    \E\left[ |Y_T^{\e,'}-Y_T^\e|^2 \right]\leq & 2 \E\left[ |\nabla g(X_T^{\e,'},\Law(X_T^{\e,'}))-\nabla g(X_T^\e,\Law(X_T^\e))|^2 \right]\\
    &+ 2 \E\left[ \left| \tilde{\E}\left[\partial_\mu g(\tilde{X}^{\e,'}_T,\Law(\tilde{X}^{\e,'}_T))(X^{\e,'}_T)-\partial_\mu g(\tilde{X}^\e_T,\Law(\tilde{X}^\e_T))(X^\e_T)\right]\right|^2 \right]\\
    \leq & 8L^2\E\left[ \sup_{t\in(0,T)}|X_t^{\e,'}-X_t^\e|^2\right].
\end{align*}

Second, for the part involving the running cost $f$, by using Young's inequality and \ref{ass:regularitycost} we obtain
\begin{align}
    2\E\left[ \int_t^T \left(\nabla_x f(s,X^{\e,'}_s,\Law(X^{\e,'}_s))-\nabla_x f(s,X^\e_s,\Law(X^\e_s)) \right)\cdot(Y_s^{\e,'}-Y^\e_s) \de s \right]\nonumber\\
    \leq \E\left[ \int_t^T|Y_s^{\e,'}-Y_s^\e|^2\de t\right]+2L^2T\E\left[ \sup_{t\in(0,T)}|X_t^{\e,'}-X_t^\e|^2\right],
\end{align}
and
\begin{align}
    2\E\left[ \int_t^T \left(\tilde{\E}\left[ \partial_\mu f(s,\tilde{X}^{\e,'}_s,\Law(\tilde{X}^{\e,'}_s))(X^{\e,'}_s)-\partial_\mu f(s,\tilde{X}^\e_s,\Law(\tilde{X}^\e_s))(X^\e_s)  \right]\right)\cdot(Y_s^{\e,'}-Y^\e_s) \de s\right]\nonumber\\
    \leq \E\left[ \int_t^T|Y_s^{\e,'}-Y_s^\e|^2\de t\right]+2L^2T\E\left[ \sup_{t\in(0,T)}|X_t^{\e,'}-X_t^\e|^2\right].
\end{align}
Last, we consider the part involving $\nabla_x b$ (the one which involves $\partial_\mu b$ works similar). We add and subtract the quantity $\nabla_x b(s,X^{\e,'}_s,\Law(X^{\e,'}_s))Y^\e_s$, then write
\begin{align*}
    2\E \left[ \int_t^T (\nabla_x b(s,X^{\e,'}_s,\right.
    & \left.\Law(X^{\e,'}_s))Y^{\e,'}_s-\nabla_x b(t,X^\e_s,\Law(X^\e_s)) Y^\e_s)\cdot(Y_s^{\e,'}-Y^\e_s) \de s\right]\\
    =2\E \left[ \int_t^T \right.&\left. (\nabla_x b(s,X^{\e,'}_s,\Law(X^{\e,'}_s))-\nabla_x b(t,X^\e_s,\Law(X^\e_s))) Y^\e_s\cdot(Y_s^{\e,'}-Y^\e_s) \de s\right]\\
    +2\E \biggl[ \int_t^T &\nabla_x b(s,X^{\e,'}_s,\Law(X^{\e,'}_s)) ( Y^{\e,'}_s-Y^\e_s)\cdot(Y_s^{\e,'}-Y^\e_s) \de s\biggr]
\end{align*}
and by using Young's inequality we simply estimate as follows
\begin{align*}
    2\E \left[ \int_t^T (\nabla_x b(s,X^{\e,'}_s,\right.
    & \left.\Law(X^{\e,'}_s))Y^{\e,'}_s-\nabla_x b(t,X^\e_s,\Law(X^\e_s)) Y^\e_s)\cdot(Y_s^{\e,'}-Y^\e_s) \de s\right]\\
    &\leq  2\E\left[ \int_t^T|Y_s^{\e,'}-Y_s^\e|^2\de t\right]
    +\E\left[ \int_t^T |\nabla_x b(s,X^{\e,'}_s,\Law(X^{\e,'}_s))|^2|Y^{\e,'}_s- Y^\e_s|^2\de s\right]\\
    &+\E\left[ \int_t^T |\nabla_x b(s,X^{\e,'}_s,\Law(X^{\e,'}_s))-\nabla_x b(t,X^\e_s,\Law(X^\e_s))|^2 |Y^\e_s|^2 \de s\right]\\
    &\leq (2+L^2) \E\left[ \int_t^T|Y_s^{\e,'}-Y_s^\e|^2\de t\right]+2TL^2C_3(T,L)^2\E\left[\sup_{t\in(0,T)}|X_t^{\e,'}-X_t^\e|^2\right],
\end{align*}
where we used \ref{ass:vf4} and \eqref{est:uniform_bound_y}.
By a Gronwall's type argument we get that
\begin{equation*}
    \E\left[|Y_t^{\e,'}-Y_t^\e|^2\right]\leq C_4(T,L)\E\left[\sup_{t\in(0,T)}|X_t^{\e,'}-X_t^\e|^2\right]
\end{equation*}
for the constant $C_4(T,L)$ given in \eqref{def:C_4}.
\end{proof}

We now provide the last set of assumptions, that are convexity hypotheses on the cost $J$.
\begin{center}
    \large {\bf Assumptions (C)}
\end{center}
\begin{enumerate}[label={\rm (C\arabic*)}]
    \item \label{ass:psiconvex} The control cost $\alpha\mapsto \psi(\alpha)\in\R$ is $\lambda$-convex over $U$, with convexity constant $\lambda>0$, i.e. for all $\alpha,\alpha'\in U$ it holds
        \begin{equation*}
            \psi(\alpha')\geq \psi(\alpha)+\partial_\alpha\psi(\alpha)\cdot(\alpha'-\alpha)+ \lambda|\alpha'-\alpha|^2.
        \end{equation*}
\end{enumerate}

With the assumptions above, we can now prove the following sufficient condition on the control for optimality. The following theorem is a slight generalization of \cite[Theorem 4.7]{Carmona2015}.
\begin{thm}\label{teo:ott-st-suff}
Let $\xi\in L^2(\Omega,\mathcal{F}_0,\mathbb{P};\R^d)$, $\hat{\alpha}_t$ be an admissible control, $X^{\e}_t$ the corresponding controlled state process, and $(Y^{\e}_t,Z^{\e}_t)$ the corresponding adjoint processes, and assume that {\bf Assumptions (A), (B), (C), (J)} hold. There exists $\Lambda:=\Lambda(T,L)>0$ 
such that if $\lambda>\Lambda$ and it holds $\mathcal{L}^1\otimes\mathbb{P}$-a.e. that
\begin{equation*}
    H(t,X^\e_t,\Law(X^\e_t),Y^\e_t,\hat{\alpha}_t)=\inf_{\alpha\in U} H(t,X^\e_t,\Law(X^\e_t),Y^\e_t,\alpha),
\end{equation*}
then $\hat{\alpha}_t$ is the unique optimal control, i.e. $J^S(\hat{\bm\alpha})=\min_{{\bm \alpha}'} J^S(\bm\alpha ')$ where the minimum is computed among the admissible controls.
\end{thm}
\begin{proof}
We drop the $\e$ superscript for simplicity of notations. Let $\alpha_t'$ be an admissible control and $X'$ the associated controlled state. By computing the cost functional and using the definition of $H$ in \eqref{def:hamiltoniana_stocastico}, we have that
\begin{align}
J^S(\hat{\bm\alpha})-J^S(\bm\alpha')=&\E\left[g(X_T,\Law(X_T))-g(X_T',\Law(X_T'))\right]\nonumber\\
&+\E\left[\int_0^T (f(t,X_t,\Law(X_t))-f(t,X_t',\Law(X_t'))+(\psi(\hat{\alpha}_t)-\psi(\alpha_t'))\de t\right]\nonumber\\
=&\E\left[g(X_T,\Law(X_T))-g(X_T',\Law(X_T'))\right]\label{eq:differenza-costi-stocastici}\\
&+\E\left[\int_0^T (H(t,X_t,\Law(X_t),Y_t,\hat{\alpha}_t)-H(t,X_t',\Law(X_t'),Y_t,\alpha_t'))\de t\right]\nonumber\\
&-\E\left[\int_0^T\left( (b(t,X_t,\Law(X_t))-b(t,X_t',\Law(X_t'))+\hat{\alpha}_t-\alpha_t')\cdot Y_t \right)\de t\right].\nonumber
\end{align}
We estimate the terms involving the final cost: we add and subtract the term
$$
\nabla_x g(X_T,\Law(X_T))\cdot(X_T-X_T')+\tilde{\E}\left[\partial_\mu g(X_T,\Law(X_T))(\tilde{X}_T)\cdot(\tilde{X}_T-\tilde{X}_T')\right],
$$
and then by using Fubini's Theorem, the fact that the tilde random variables are independent copies of the non-tilde variables, the definition of $Y_T$, and Lemma \ref{lem:stimab} we have
\begin{equation}
\E\left[g(X_T,\Law(X_T))-g(X_T',\Law(X_T'))\right]\leq \E\left[Y_T\cdot(X_T-X_T')\right]+ 2L \E\left[|X_T-X_T'|^2\right]\label{est:finalcost}.
\end{equation}

We now use the adjoint equation to compute
\begin{align}
\E\left[Y_T\cdot(X_T-X_T')\right]= & \E\left[\int_0^T (X_t-X_t')\cdot\de Y_t+\int_0^T Y_t\cdot\de( X_t-X_t') \right]\nonumber\\
=&-\E\left[\int_0^T \nabla_x H(t,X_t,\Law(X_t),Y_t,\hat{\alpha}_t)\cdot(X_t-X_t') \de t\right]\nonumber\\
&-\E\left[\int_0^T \tilde{\E}\left[ \partial_\mu H(t,\tilde{X}_t,\Law(\tilde{X}_t),\tilde{Y}_t,\tilde{\hat{\alpha}}_t)(X_t) \right]\cdot(X_t-X_t') \de t\right]\label{eq:fubini1}\\
&+\E\left[ \int_0^T \left(b(t,X_t,\Law(X_t))-b(t,X_t',\Law(X_t'))+\hat{\alpha}_t-\alpha_t'\right))\cdot Y_t \de t\right].\nonumber
\end{align}
Again by Fubini's theorem, it holds
\begin{align}\label{eq:fubini2}
\E&\left[ \int_0^T\tilde{\E}\left[ \partial_\mu H(t,\tilde{X}_t,\Law(\tilde{X}_t),\tilde{Y}_t,\tilde{\hat{\alpha}}_t)(X_t) \right]\cdot(X_t-X_t') \de t\right]\\
&=\E\tilde{\E}\left[\int_0^T \partial_\mu H(t,X_t,\Law(X_t),Y_t,\hat{\alpha}_t)(\tilde{X}_t) \cdot(\tilde{X}_t-\tilde{X}_t') \de t\right]\nonumber\\
&=\E\left[ \int_0^T\tilde{\E}\left[ \partial_\mu H(t,X_t,\Law(X_t),Y_t,\hat{\alpha}_t)(\tilde{X}_t) \right]\cdot(\tilde{X}_t-\tilde{X}_t') \de t\right].
\end{align}
Then, by using \eqref{eq:differenza-costi-stocastici}, \eqref{est:finalcost}, \eqref{eq:fubini1}, and \eqref{eq:fubini2} we have 
\begin{align}
J^S(\hat{\bm\alpha})-J^S(\bm\alpha')\leq &\E\left[ \int_0^T (H(t,X_t,\Law(X_t),Y_t,\hat{\alpha}_t)-H(t,X_t',\Law(X_t'),Y_t,\alpha_t')) \de t \right]\nonumber\\
&-\E\left[\int_0^T \nabla_x H(t,X_t,\Law(X_t),Y_t,\hat{\alpha}_t)\cdot(X_t-X_t') \de t\right]\nonumber\\
&-\E\left[\int_0^T \tilde{\E}\left[ \partial_\mu H(t,X_t,\Law(X_t),Y_t,\hat{\alpha}_t))(\tilde{X}_t) \cdot(\tilde{X}_t-\tilde{X}_t')\right] \de t\right]\nonumber\\
&+ 2L \E\left[|X_T-X_T'|^2\right].\label{est:spezzareconv}
\end{align}

We estimate the right hand side of \eqref{est:spezzareconv} as follows: for the part involving the running cost we use Lemma \ref{lem:stimab} in order to obtain that
\begin{align}
\E \Big[\int_0^T &(f(t,X_t,\Law(X_t))-f(t,X_t',\Law(X_t')))- \nabla f(t,X_t,\Law(X_t))\cdot(X_t-X_t')\de t\Big]\nonumber\\
&-\E\left[\int_0^T \tilde{\E}\left[\partial_\mu f(t,X_t,\Law(X_t))(\tilde{X}_t)\cdot(\tilde{X}_t-\tilde{X}_t') \right]\de t \right]\leq 2L\E\left[\int_0^T|X_t-X_t'|^2\de t\right].\label{eq:modifica-f}
\end{align}

Then, we estimate the part involving the vector field $b$: we apply Lemma \ref{lem:stimab}, Lemma \ref{lem:constant-fsde} and Lemma \ref{lem:constant-bsde} to obtain
\begin{align*}
    \E \Bigg[\int_0^T Y_t\cdot \Bigg(b(t,X_t,\Law(X_t))&-b(t,X_t',\Law(X_t'))-\nabla b(t,X_t,\Law(X_t))(X_t-X_t')\Bigg)\de t \Bigg]\\
    &-\E\Bigg[ \int_0^T Y_t\cdot\Bigg(\tilde{\E}\Bigg[\partial_\mu b(t,X_t,\Law(X_t))(\tilde{X}_t)(\tilde{X}_t-\tilde{X}_t') \Bigg]\Bigg)\de t \Bigg]\\
    &\leq \Lambda_1(T,L)\left(\E\left[ \int_0^T|\hat{\alpha}_t-\alpha_t'|^2\de t \right]\right),
\end{align*}
where
\begin{equation}\label{def:rmktempi}
    \Lambda_1(T,L):=2LT C_2(T,L)C_3(T,L).
\end{equation}
Finally, for the part involving the control cost, thanks to minimality of $\hat{\alpha}_t$ and convexity of $\psi$, it holds
\begin{equation}
    \E\left[ \int_0^T\psi(\hat{\alpha}_t)-\psi(\alpha_t') + (\hat{\alpha}_t-\alpha_t')\cdot Y_t \de t \right]\leq -\lambda\E\int_0^T|\hat{\alpha}_t-\alpha_t'|^2\de t.
\end{equation}
In conclusion, we use Lemma \ref{lem:constant-fsde} and define
\begin{equation}\label{def:Lambda}
\Lambda:=\Lambda_1(T,L)+2LC_2(T,L)T+2L\tilde C_2(L,T)T, 
\end{equation}
where we used \eqref{eq:further-estimate} to estimate the contribution of the final cost. Finally, we have that
\begin{equation}\label{stimaJmin}
    J^S(\hat{\bm\alpha})+(\lambda-\Lambda)\E\left[\int_0^T|\hat{\alpha}_t-\alpha'_t|^2\de t\right]\leq J^S(\bm\alpha'),
\end{equation}
which in turn gives that $\hat{\bm\alpha}$ is the unique optimal control if $\lambda>\Lambda$.
\end{proof}

\begin{rem}\label{rem:tempi-piccoli}
Note that the constant $\Lambda=O(T)$ in \eqref{def:Lambda} can be made as small as desired for $T$ small enough. This implies that, at least for small times, the condition $\lambda>\Lambda$ is always satisfied. A similar condition also ensures Lipschitz regularity of mean-field optimal control (see \cite{LipReg}) and uniqueness of a minimizer in mean-field games (see \cite{bardi-fischer}). Note that another possible choice for the lower bound on the convexity constant is
$$
\tilde\Lambda:=\Lambda_1(T,L)+2LC_2(T,L)(1+T),
$$
which is smaller than $\Lambda$ for very large $T$.
\end{rem}

\begin{rem}\label{rem:convexityfg}
It is clear that a convexity assumption on $f$ and $g$ simplifies the proof of Theorem \ref{teo:ott-st-suff}: in this case the bounds \eqref{est:finalcost} and \eqref{est:spezzareconv} hold without the terms 
$$
\E\left[\int_0^T|X_t-X_t'|^2\de t\right],\qquad
\E\left[|X_T-X_T'|^2\right],
$$
respectively. In this case the theorem holds if we assume $\lambda>\Lambda_1(T,L)$, which turns out to be a constant smaller than $\Lambda$. At this point, as one could expect, it is worth to note that a strict convexity assumption on $f$ and $g$ would allow to consider a larger class of control costs $\psi$ with a smaller convexity constant, with a condition of the type
$$
\lambda+\lambda_g+\lambda_f-\Lambda>0.
$$
However, we cannot avoid to consider $\lambda>0$: it will be clear from the Lemma \ref{lem:lipcontr} below and the counterexample of Section \ref{s-controesempio}.
\end{rem}

\begin{rem}\label{rem:baffine}
Note that, if $f,g$ are convex and $b$ is affine (eventually depending on the barycenter of $\mu$ too), i.e. of the form
$$
b(t,x,\mu)=b_0(t)+b_1(t)x+b_2(t)\int_{\R^d}x\,\mu(\de x),
$$
then {\bf Assumption (C)} implies that the function $(x,\mu,\alpha)\mapsto H(t,x,\mu,y,\alpha)$ is convex. The proof of Theorem \ref{teo:ott-st-suff} then follows in a simpler way, without resorting to Lemma \ref{lem:stimab}, see \cite{CarmonaDelarueI}. In particular, under these hypotheses we can eventually set $\Lambda=0$.
\end{rem}

We recall here the necessary condition for optimality for the stochastic problem, see \cite{CarmonaDelarueI}[Theorem 6.14]. We will use this theorem in the proof of the corollary.

\begin{thm}\label{thm:cond-necessaria}
Assume $b,f,\psi,g$ satisfy {\bf Assumptions (B), (C), (J)}. Then, if we assume that $\mathcal{F}_t$ is generated by $\mathcal{F}_0$ and $W_t$, that the Hamiltonian $H$ is convex in $\alpha\in U$ that the admissible control $\bm\alpha$  is optimal, that $X$ is the associated (optimally) controlled state, and that $(Y, Z)$ are the associated adjoint processes solving \eqref{eq:adj_st-generica}, then we have:
\begin{equation}
    \forall\alpha \in U,\qquad H(t,X_t,\Law(X_t),Y_t,Z_t,\alpha_t)\leq H(t,X_t,\Law(X_t),Y_t,Z_t,\alpha),
\end{equation}
$\mathscr{L}^1\otimes\mathbb{P}$ almost everywhere.
\end{thm}

We now show that the optimal control is Lipschitz continuous when the Hamiltonian is strictly convex with respect to the control.

\begin{lem}\label{lem:lipcontr}
Under {\bf Assumptions (A), (B), (C), (J)}, there exists a unique minimizer  $\hat{\alpha}$ of $H$. Moreover, the function $\hat{\alpha}:y\in\R^d\to\hat{\alpha}(y)\in U$ is measurable and Lipschitz continuous, with a Lipschitz constant depending on $\lambda$ only. 
\end{lem}
\begin{proof}
Observe that, for any $(t,x,\mu,y)$, the function $\alpha\mapsto H(t,x,\mu,y,\alpha)$ is continuously differentiable and strictly convex, so that $\hat{\alpha}(y)$ appears as the unique solution of the variational inequality:
\begin{equation}
    \forall \beta\in U,\qquad\mbox{~~~it holds~~~}\qquad\left( \hat{\alpha}(y)-\beta)\cdot (y+\nabla\psi(\hat{\alpha})\right)\leq 0.
\end{equation}
Moreover, by strict convexity, measurability of $\hat{\alpha}(y)$ is a consequence of the {\em gradient descent} algorithm with convex constraint, see \cite{CarmonaDelarueI}.
We now prove Lipschitz continuity: for any $y,y' \in\R^d$ and $(t,x,\mu)\in[0,T]\times\R^d \times\Pn$, the criticality of $\hat{\alpha}$ provides the following inequalities
\begin{equation}\label{ineq:1}
\left( \hat{\alpha}(y)-\hat{\alpha}(y')\right)\cdot\partial_\alpha H(t,x,\mu,y',\hat{\alpha}(y'))\geq 0, 
\end{equation}
\begin{equation}\label{ineq:2}
\left( \hat{\alpha}(y')-\hat{\alpha}(y)\right)\cdot\partial_\alpha H(t,x,\mu,y,\hat{\alpha}(y)) \geq 0.
\end{equation}
They in turn imply
\begin{equation}\label{est:alphaH}
\left( \hat{\alpha}(y')-\hat{\alpha}(y)\right)\cdot\left(\partial_\alpha H(t,x,\mu,y',\hat{\alpha}(y'))-\partial_\alpha H(t,x,\mu,y,\hat{\alpha}(y)) \right)\leq 0.
\end{equation}
On the other hand, by using $\lambda$-convexity of $\psi$, we also have
\begin{equation*}
(\alpha'-\alpha)\cdot(\partial_\alpha\psi(\alpha')-\partial_\alpha\psi(\alpha))\geq 2\lambda|\alpha'-\alpha|^2.
\end{equation*}
Since $\partial_\alpha H=y+\partial_\alpha \psi$, we also have
\begin{align*}
2\lambda|\hat{\alpha}(y')-\hat{\alpha}(y)|^2\leq & (\hat{\alpha}(y')-\hat{\alpha}(y))\cdot(\partial_\alpha\psi(\hat{\alpha}(y'))-\partial_\alpha\psi(\hat{\alpha}(y))) \\
\leq &\left( \hat{\alpha}(y')-\hat{\alpha}(y)\right)\cdot\left(\partial_\alpha H(t,x,\mu,y',\hat{\alpha}(y'))-\partial_\alpha H(t,x,\mu,y,\hat{\alpha}(y)) \right)\\
&+\left( \hat{\alpha}(y')-\hat{\alpha}(y)\right)\cdot(y'-y) \leq |\hat{\alpha}(y')-\hat{\alpha}(y)||y'-y|.
\end{align*}
In the last inequality we have used \eqref{est:alphaH}. It then follows
\begin{equation}
|\hat{\alpha}(y')-\hat{\alpha}(y)|\leq \tfrac{1}{2\lambda}|y'-y|.
\end{equation}
This concludes the proof.
\end{proof}

\begin{rem}
From the Lemma above we deduce that the minimum $\hat{\alpha}$ of the Hamiltonian $H$ does not depend on $\e$. On the other hand, the optimal control depends on $\e$, via the adjoint process: it can be written as follows
\begin{equation}\label{def:ottimo-ye}
\hat{\alpha}^\e_t:=\hat{\alpha}(Y_t^\e).
\end{equation}
\end{rem}

Given the optimal control $\hat{\alpha}^\e_t$, we can write a {\em forward backward system of stochastic differential equations} (FB-SDE) of McKean-Vlasov type, that is
\begin{equation}\label{eq:FB-PMP}\tag{FB-SDE}
\begin{cases}
\de X^\e_t=\left( b(t,X^\e_t,\mathcal{L}(X^\e_t))+\hat{\alpha}(Y^\e_t) \right) \de t +\sqrt{2\e}\,\de W_t,\\
\de Y^\e_t=-(\nabla_x H(t,X^\e_t,\mathcal{L}(X^\e_t),Y^\e_t,\hat{\alpha}(Y^\e_t))+\tilde{\E}[\partial_\mu H(t,\tilde{X}^\e_t,\mathcal{L}(\tilde{X}^\e_t),\tilde{Y}^\e_t,\hat{\alpha}(\tilde{Y}^\e_t))(X^\e_t)])\de t\\
\hspace{1.5cm}+Z^\e_t\de W_t,\\
X^\e_0=\xi,\hspace{0.5cm}
Y^\e_T=\nabla_x g(X^\e_T,\mathcal{L}(X^\e_T))+ \tilde{\E}[\partial_\mu g(\tilde{X}^\e_T,\mathcal{L}(\tilde{X}^\e_T))(X_T)],
\end{cases}
\end{equation}
where we recall again that the notation $(\tilde{X}^\e_t,\tilde{Y}^\e_t)$ denotes an independent copy of $(X^\e_t,Y^\e_t)$ defined on the space $(\tilde{\Omega},\tilde{\mathcal{F}}_t,\tilde{\mathbb{P}})$ and $\tilde{\E}$ denotes the expectation on $(\tilde{\Omega},\tilde{\mathcal{F}}_t,\tilde{\mathbb{P}})$. If the coefficients are smooth and no further assumptions are required, systems of FB-SDE are not always solvable, see \cite{Antonelli93}. For a fixed $\e>0$, existence of a solution of \eqref{eq:FB-PMP} is provided in \cite{Carmona2013b}. In the next subsection, we show that convexity of the Hamiltonian ensures well-posedness of \eqref{eq:FB-PMP}, even when the viscosity coefficient is zero. Thus, stability estimates on the solutions of \eqref{eq:FB-PMP} will turn into an $\e$-uniform bound on the Lipschitz constant of the optimal control of $\Ppe$, as we will show in Section 4.

\subsection{Well-posedness of (FB-SDE)}
The goal of this subsection is to prove well-posedness of the system \eqref{eq:FB-PMP} associated to an optimal control, and then to build the associated decoupling field. Note that Theorem \ref{teo:ott-st-suff} ensures that solving (SOC) is equivalent to solve the system \eqref{eq:FB-PMP}. We drop the $\e$ superscript in the whole subsection, for simplicity of notations.

We adopt the strategy known as {\em continuation method} for FB-SDEs, see  \cite{PengWu99}. We denote by $\Theta_t:=(X_t,\Law(X_t),Y_t,\hat{\alpha}_t)$, where $\hat{\alpha}_t=\hat{\alpha}(Y_t)$, and $\mathscr{S}$ is the space of the processes $\Theta_t$ such that
\begin{equation}
\|\Theta\|_\mathscr{S}:=\E\left[ \sup_{t\in(0,T)}\left(|X_t|^2+|Y_t|^2\right)+\int_0^T\left(|Z_t|^2+|\hat{\alpha}_t|^2\right)\de t \right]^{1/2} <+\infty,
\end{equation}
where $Z_t$ is the process associated to $Y_t$ as in \eqref{eq:FB-PMP}. Similarly, we define $\theta_t:=(X_t,\Law(X_t))$. Moreover, an {\em input} $\mathcal{I}=(I_t^b,I_t^\sigma,I_t^f,I_T^g)$ will be a four-tuple where the first three entries are square-integrable progressively measurable processes and the last one is an $\mathcal{F}_T$ square-integrable random variable. We denote by $\mathbb{I}$ the space of inputs endowed with the norm
\begin{equation}
\|I\|_\mathbb{I}:=\E\left[ |I_T^g|^2+\int_0^T\left(|I_t^b|^2+|I_t^\sigma|^2+|I_t^f|^2\right)\de t \right]^{1/2}<+\infty.
\end{equation}
\begin{defn}
For each $\gamma\in [0,1], \xi\in L^2(\Omega,\mathcal{F}_0,\mathbb{P};\R^d)$ and $I\in \mathbb{I}$, define $\mathcal{E}(\gamma,\xi,I)$ as the FB-SDE:
\begin{equation}\label{eq:FBSDE-gamma}
\begin{cases}
\de X_t=\left(\gamma [b(t,\theta_t)+\hat{\alpha}_t]+I_t^b\right)\de t+\left(\gamma\sqrt{2\e}+I_t^\sigma\right)\de W_t,\\
\de Y_t=-\left( \gamma\left\{ \nabla_x H(t,\Theta_t)+\tilde{\E}\left[ \partial_\mu H(t,\tilde{\Theta}_t)(X_t) \right]\right\}+I_t^f\right)\de t+Z_t\de W_t,\\
X_0=\xi,\\
Y_T=\gamma\left\{\nabla g(X_T,\Law(X_T))+\tilde{\E}\left[ \partial_\mu g(\tilde{X}_T,\Law(\tilde{X}_T))(X_T)\right]\right\}+I_T^g.
\end{cases}
\end{equation}
For any $\gamma\in[0,1]$, we say that the property $(S_\gamma)$ holds if, for any $\xi\in L^2(\Omega,\mathcal{F}_0,\mathbb{P};\R^d)$ and $I\in \mathbb{I}$, the FB-SDE $\mathcal{E}(\gamma,\xi,I)$ has a unique solution in $\mathscr{S}$.
\end{defn}

We now provide a stability lemma for solutions of \eqref{eq:FBSDE-gamma}.  It is very similar to  the one of \cite[Lemma 5.5]{Carmona2015} and generalizes it to non affine vector-fields.
\begin{lem}\label{lem:stabilityTheta}
Let $\gamma\in [0,1]$ such that $(S_\gamma)$ holds. Then, there exists a constant $C$, which depends on $T,L,\lambda$ and it is independent on $\gamma$ and $\e$, such that for any $\xi,\xi'\in L^2(\Omega,\mathcal{F}_0,\mathbb{P};\R^d)$ and $I,I'\in \mathbb{I}$, the solutions $\Theta,\Theta'$ of $\mathcal{E}(\gamma,\xi,I), \mathcal{E}(\gamma,\xi',I')$ satisfy:
\begin{equation}
\|\Theta-\Theta'\|_{\mathscr{S}}^2\leq C\left(\E\left[|\xi-\xi'|^2\right]+\|I-I'\|_\mathbb{I}^2\right).
\end{equation}
\end{lem}
\begin{proof}
The proof strongly relies on the estimates (and the strategies) proved in Lemma \ref{lem:constant-fsde} and Lemma \ref{lem:constant-bsde}. For the forward component $X_t$ we have 
\begin{equation}\label{est:FSDE}
\E\left[ \sup_{t\in(0,T)}|X_t-X_t'|^2 \right]\leq \E\left[ |\xi-\xi'|^2 \right]+C\gamma\E\left[ \int_0^T|\hat{\alpha}(Y_t)-\hat{\alpha}(Y_t')|^2\de t \right]+C\|I-I'\|^2_\mathbb{I},
\end{equation}
while for the backward component $Y_t,Z_t$ it holds
\begin{align}
\E&\left[\sup_{t\in(0,T)}|Y_t-Y_t'|^2 +\int_0^T|Z_t-Z_t '|^2\de t\right]\nonumber\\
&\leq C\gamma\E\left[ \sup_{t\in(0,T)}|X_t-X_t'|^2+\int_0^T|\hat{\alpha}(Y_t)-\hat{\alpha}(Y_t')|^2\de t \right] +C\|I-I'\|^2_\mathbb{I}.\label{est:BSDE}
\end{align}
In order to close the estimates, we follow the computations of Theorem \ref{teo:ott-st-suff} to obtain a bound on the terms involving the optimal control. First of all, by using the estimate in \eqref{est:finalcost} we have that
\begin{align*}
    \E\left[ (X_T'-X_T)\cdot Y_T \right]&\leq \gamma\E\left[g(X'_T,\Law(X_T')-g(X_T,\Law(X_T))\right]+2\gamma L \E\left[|X_T-X_T'|^2\right]\\
    &+\E\left[ (X_T'-X_T)\cdot I^{g}_T \right]
\end{align*}

On the other hand, by using the equations we have that
\begin{align}
\E&\left[ (X_T'-X_T)\cdot Y_T \right]= \E\left[(\xi'-\xi)\cdot Y_0 \right]-\gamma\E\left[\int_0^T \nabla_x H(t,X_t,\Law(X_t),Y_t,\hat{\alpha}_t)\cdot(X'_t-X_t) \de t\right]\nonumber\\
&-\gamma\E\left[ \int_0^T\tilde{\E}\left[ \partial_\mu H(t,X_t,\Law(X_t),Y_t,\hat{\alpha}_t)(\tilde{X}_t) \right]\cdot(\tilde{X}_t'-\tilde{X}_t) \de t\right]\nonumber\\
&+\gamma\E\left[ \int_0^T \left(b(t,X_t',\Law(X_t'))-b(t,X_t,\Law(X_t))+\hat{\alpha}_t'-\alpha_t\right))\cdot Y_t \de t\right]\nonumber\\
&-\E\int_0^T\left[(X_t'-X_t)\cdot I^f_t+(I^{b}_t-I^{b,'}_t)\cdot Y_t+ (I^\sigma_t-I^{\sigma,'}_t)\cdot Z_t\right]\de t= T_0-\gamma T_1-T_2.\nonumber
\end{align}
where we used Fubini's theorem to switch the integrals in the second line. Then, by repeating the proof of Theorem \ref{teo:ott-st-suff} we obtain that
\begin{equation}\label{eq:reverse}
    \gamma J^S(\hat{\alpha}')-\gamma J^S(\hat{\alpha})\geq \gamma (\lambda-\Lambda(T,L))\E\left[\int_0^T|\hat{\alpha}_t-\hat{\alpha}_t'|^2\right]\de t+T_0-T_2+\E\left[ (X_T-X_T')\cdot I^{g}_T \right],
\end{equation}
where $\Lambda(T,L)$ is the constant defined in \eqref{def:Lambda}.
We now reverse the role of $\hat{\alpha}_t$ and $\hat{\alpha}_t'$ in \eqref{eq:reverse}, and by denoting with $T_0'$ and $T_2'$ the corresponding terms in the inequality, we obtain that
\begin{equation}
    2\gamma (\lambda-\Lambda(T,L))\E\left[|\hat{\alpha}_t-\hat{\alpha}_t'|^2\right]\de t+T_0+T_0'-T_2-T_2'+\E\left[ (I^{g}_T-I^{g,'}_T)\cdot (X_T-X_T') \right]\leq 0.
\end{equation}
We remark that the condition above follows from the convexity assumptions and not from the optimality of $\hat{\alpha}_t'$ or $\hat{\alpha}_t$, indeed the corresponding trajectories start from a different initial condition thus they are respectively not admissible. Then, since
$$
T_0+T_0'=\E\left[(\xi'-\xi)\cdot (Y_0-Y_0') \right],
$$
and similarly
\begin{align*}
    T_2+T_2'&=\E\int_0^T\left[-(X_t-X_t')\cdot (I^f_t-I^{f,'}_t)+(I^{b}_t-I^{b,'}_t)\cdot (Y_t-Y_t')\right]\de t\\
    &+ \E\int_0^T\left[(I^\sigma_t-I^{\sigma,'}_t)\cdot (Z_t-Z_t')\right]\de t,
\end{align*}
by using Young inequality we have that
\begin{equation}\label{est:controlli}
    \gamma (\lambda-\Lambda(T,L))\E\left[|\hat{\alpha}_t-\hat{\alpha}_t'|^2\right]\de t\leq \eta\|\Theta-\Theta'\|_{\mathscr{S}}^2+\frac{C}{\eta} \left(\E\left[|\xi-\xi'|^2\right]+\|I-I'\|_\mathbb{I}^2\right).
\end{equation}
Finally, by using that $\lambda>\Lambda(T,L)$ and by plugging \eqref{est:controlli} in \eqref{est:FSDE} and \eqref{est:BSDE} we obtain that
\begin{align}
    \E\left[\sup_{t\in(0,T)}|X_t-X_t'|^2\right.&+\sup_{t\in(0,T)}|Y_t-Y_t'|^2 \left.+\int_0^T|Z_t-Z_t '|^2\de t\right]\nonumber\\
    &\leq C\eta\|\Theta-\Theta'\|_{\mathscr{S}}^2+\frac{C}{\eta} \left(\E\left[|\xi-\xi'|^2\right]+\|I-I'\|_\mathbb{I}^2\right),
\end{align}
where it is worth to note that the constant $C$ blows up as $\lambda\to\Lambda(T,L)$. Then, the result follows by using the Lipschitz continuity of $\hat{\alpha}$ and choosing $\eta$ small enough.
\end{proof}

\begin{rem}
Note that the proof of Lemma \ref{lem:stabilityTheta} is similar to the proof of Theorem \ref{teo:ott-st-suff}. Thus, the Remark \ref{rem:convexityfg} also holds for it.
\end{rem}

We now give an {\em induction} lemma for the system \eqref{eq:FBSDE-gamma}.
\begin{lem}\label{lem:induction}
There exists a $\delta_0>0$, which depends on $T,L,\lambda$ only, such that, if $(S_\gamma)$ holds for some $\gamma\in[0,1)$, then $(S_{\gamma+\eta})$ holds for all $\eta\in(0,\delta_0]$ satisfying $\gamma+\eta\leq 1$.
\end{lem}
\begin{proof}
The proof follows a standard Picard's contraction argument. Indeed, if $\gamma$ is such that $(\mathcal{S}_\gamma)$ holds, for $\eta>0,\xi\in L^2(\Omega,\mathscr{F}_0,\mathbb{P};\R^d)$ and $I\in\mathbb{I}$, we define the map $\Phi:\mathscr{S}\to\mathscr{S}$ whose fixed points coincide with the solution of $\mathcal{E}(\gamma+\eta,\xi, I)$. We now give the definition of $\Phi$. Given a process $\Theta\in \mathscr{S}$, we denote with $\Theta'$ the solution of $\mathcal{E}(\gamma,\xi,I')$ with
\begin{align*}
I_t^{b,'}&=\eta b(t,\theta_t)+\eta\hat{\alpha}(Y_t)+I_t^{b}\\
I_t^{f,'}&=\eta\nabla H(t,\Theta_t)+\eta\tilde{\E}\left[ \partial_\mu H(t,\tilde{\Theta}_t)(X_t) \right]+I^f_t\\
I_t^{\sigma,'}&=\eta \sqrt{2\e}+I^\sigma_t\\
I_T^{g,'}&=\eta\nabla g(X_T,\Law(X_T))+\eta\tilde{\E}\left[ \partial_\mu g(\tilde{X}_T,\Law(\tilde{X}_T))(X_T)\right]+I^g_T.
\end{align*}
By assumptions, it is uniquely defined and it belongs to $\mathscr{S}$, so the mapping $\Phi:\Theta\to\Theta '$ maps $\mathscr{S}$ into itself. It is clear that a process $\Theta$ is a fixed point for $\Phi$ if and only if $\Theta$ is a solution of $\mathcal{E}(\gamma,\xi,I')$. We now only need to prove that $\Phi$ is a contraction when $\eta$ is small enough. Given $\Theta^1,\Theta^2\in\mathscr{S}$, by Lemma \ref{lem:stabilityTheta} we get that
\begin{equation}
\| \Phi(\Theta^1)-\Phi(\Theta^2) \|_{\mathscr{S}}\leq C\eta\| \Theta^1-\Theta^2 \|_{\mathscr{S}},
\end{equation}
which is enough to conclude the proof, since the constant $C$ does not depend on $\gamma$.
\end{proof}

We are now able to prove well-posedness of \eqref{eq:FB-PMP}.

\begin{thm}\label{thm:well-posedness fb-sde}
Assume that {\bf Assumptions (A), (B), (C), (J)} hold. Then, for any initial $\xi\in L^2(\Omega,\mathcal{F}_0,\mathbb{P};\R^d)$, the system \eqref{eq:FB-PMP} is uniquely solvable.
\end{thm}
\begin{proof}
First, note that for $\gamma=0$, the right hand side of the system \eqref{eq:FBSDE-gamma} is made up of square-integrable progressively measurable processes, and it does not depends on the solution itself. So $(S_0)$ obviously holds. Then, the proof is a straightforward  induction argument based on Lemma \ref{lem:induction}.
\end{proof}

As already stressed above, by Theorem \ref{thm:well-posedness fb-sde} we know that the solution of \eqref{eq:FB-PMP} is the unique optimal path of the stochastic control problem (SOC). 
To conclude this section, we show the existence of a decoupling field related to the system \eqref{eq:FB-PMP}, which will allow to write the optimal control $\hat{\alpha}_t^\e$ in feedback form.

\begin{lem}\label{lem:esistenza_master-field}
For any $t\in[0,T]$ and $\xi\in L^2(\Omega,\mathcal{F}_t,\mathbb{P};\R^d)$, there exists a unique solution $$(X^{\xi,\e}_{t,s},Y^{\xi,\e}_{t,s},Z^{\xi,\e}_{t,s})_{t\leq s\leq T}$$ of the system \eqref{eq:FB-PMP} on $[t,T]$ with $X^{\xi,\e}_{t,t}=\xi$. Moreover, for any $\mu\in \Pn$, there exists a measurable mapping $\U^\e(t,\cdot,\mu):x\in\R^d\mapsto\U^\e(t,x,\mu)$ such that:
\begin{equation}
\mathbb{P}\left[Y_{t,t}^{\xi,\e}=\U^\e(t,\xi,\mathcal{L}(\xi))\right]=1.
\end{equation}
Moreover, there exists a constant $C$, depending only on the parameters in {\bf Assumptions (A), (B), (C), (J)}, such that, for any $t\in [0,T]$ and any $\xi_1,\xi_2\in L^2(\Omega,\mathcal{F}_0,\mathbb{P};\R^d)$,
\begin{equation}\label{def:lips_constant-master-field}
\E\left[ |\U(t,\xi_1,\mathcal{L}(\xi_1))-\U(t,\xi_2,\mathcal{L}(\xi_2))|^2 \right]\leq C\E\left[ |\xi_1-\xi_2|^2\right].
\end{equation}
\end{lem}
\begin{proof}
Given $t\in[0,T)$ and $\xi\in L^2(\Omega,\mathscr{F}_t,\mathbb{P};\R^d)$, existence and uniqueness of a solution of \eqref{eq:FB-PMP} on $[t,T]$ with initial condition $\xi$ is a direct consequence of Theorem \ref{thm:well-posedness fb-sde}. We now proceed to define the decoupling field. First of all, note that $Y_{t,t}^{\xi,\e}$ coincide a.s. with a $\sigma\{\xi\}$-measurable $\R^d$-valued random variable. In particular, there exists $u_\xi^\e(t,\cdot):\R^d\to\R^d$ such that $\mathbb{P}\left[Y_{t,t}^{\xi,\e}=u_\xi^\e(t,\xi)\right]=1$. Moreover, the law of $(\xi,Y_{t,t}^{\xi,\e})$ only depends on the law of $\xi$, as a consequence of Yamada-Watanabe Theorem, see \cite[Theorem 1.33]{CarmonaDelarueII}. Since uniqueness holds pathwise, it also holds in law, so given two initial conditions with the same law, the solution  has the same law. Therefore, given another $\R^d$-valued random vector $\xi'$ with the same law of $\xi$, it holds that $(\xi,u^\e_\xi(t,\xi))$ has the same law of $(\xi',u^\e_{\xi'}(t,\xi'))$. In particular, for any measurable vector field $v:\R^d\to\R^d$, the random variables $u_\xi^\e(t,\xi)-v(\xi)$ and $u^\e_{\xi'}(t,\xi')-v(\xi')$ have the same law. Choosing $v=u_\xi^\e(t,\cdot)$ we deduce that $u^\e_\xi(t,\cdot)$ and $u^\e_{\xi'}(t,\cdot)$ are equal a.e. under the same law $\Law(\xi)$. This means that, by denoting $\Law(\xi)=\mu$, there exists an element $\U^\e(t,\cdot,\mu)\in L^2(\R^d;\mu)$ such that both $u_\xi^\e(t,\cdot)$ and $u_{\xi'}^\e(t,\cdot)$ coincide with $\U^\e(t,\cdot,\mu)$. Identifying $\U^\e(t,\cdot,\mu)$ with one of its versions, we have that $\mathbb{P}\left[ Y_{t,t}^{\xi,\e}=\U^\e(t,\xi,\mu) \right]=1$. When $t>0$, for any $\mu \in \Pn$ there exists a $\mathcal{F}_t$-measurable random variable $\xi$ with law $\mu$. As a consequence, this procedure allows to define $\U^\e(t,\cdot,\mu)$ for any $\mu\in\Pn$. Note that, when $t=0$, $\mathcal{F}_0$ is not trivial, i.e. it does not reduce to events of measure zero or one, since we assume that $(\Omega,\mathcal{F}_0, \mathbb{P})$ is atomless; thus it supports $\R^d$-valued random variables with arbitrary distributions, see discussion at the beginning of Section \ref{s-assJ}.


The fact that $\U^\e$ is independent from the probabilistic set-up $(\Omega,\mathcal{F}_t,\mathbb{P})$ directly follows from the uniqueness in law.

Finally, the Lipschitz property of $\mathcal{U}^\e(0,\cdot,\cdot)$ is a consequence of Lemma \ref{lem:stabilityTheta} with $\gamma=1$. Shifting time if necessary, the same argument applies to $\mathcal{U}^\e(t,\cdot,\cdot)$.
\end{proof}

\begin{rem}
It is worth to notice that the decoupling fields are different if the laws of the initial conditions are different.
\end{rem}

\section{The vanishing viscosity method}
In this section we prove our main result. We first build the optimal control for $\Ppe$ using the theory developed in Section 3. Then, we will provide some convergence lemma, which will be the core of the proof of Theorem \ref{thm:main}. To make the presentation smoother, we will always assume that {\bf Assumptions (A), (B), (C), (J)} hold, without recalling them.

\subsection{The viscous optimal control}
Let $\U^\e$ be the decoupling field given by Lemma \ref{lem:esistenza_master-field}. Thanks to Proposition \ref{prop:lip_function_measure}, for any $\mu\in\Pn$, we can consider a version of $x\mapsto\U^\e(t,x,\mu)$ in $L^2(\R^d,\mu)$ that is Lipschitz continuous with respect to $x$, for the same Lipschitz constant $C$ as in \eqref{def:lips_constant-master-field}. This is crucial for what follows.
\begin{lem}\label{lem:def-ottimo-viscoso}
Let $\hat{\alpha}$ be the minimizer of the Hamiltonian \eqref{def:hamiltoniana_stocastico} in $U$, and $\U^\e$ the decoupling field defined in Lemma \ref{lem:esistenza_master-field}. Then, the map $u^\e:[0,T]\times\R^d\to\R^d$ defined by
\begin{equation}\label{def:ottimo-viscoso}
u^\e(t,x)=\hat{\alpha}(\U^\e(t,x,\mu^\e_t)),
\end{equation}
is the unique optimal control for $\Ppe$. Moreover, the control $u^\e$ is Lipschitz continuous, uniformly with respect to time and viscosity coefficient: i.e., there exists a constant $L_{\lambda}>0$ independent on $t$ and $\e$ such that
\begin{equation}\label{def:costante-lip-ottimo-vis}
|u^\e(t,x)-u^\e(t,x')|\leq L_\lambda |x-x'|.
\end{equation}
\end{lem}
\begin{proof}
First remark that Lipschitz continuity follows from Lemma \ref{lem:lipcontr}, Lemma \ref{lem:esistenza_master-field} and Proposition \ref{prop:lip_function_measure} it holds
\begin{equation}
|u^\e(t,x)-u^\e(t,x')|= |\hat{\alpha}(\U^\e(t,x,\mu^\e_t))-\hat{\alpha}(\U^\e(t,x',\mu^\e_t))|\leq L_\lambda |x-x'|,
\end{equation}
for some constant $L_\lambda$ which only depends on the norms of $b,f,g$ and on $\lambda$. We now prove optimality. Let $w:[0,T]\times\R^d\to U$ be an admissible control for $\Ppe$. Then, the vector field $\alpha_t$ defined as
$$
\alpha_t=w(t,X^{w,\e}_t)
$$
is an admissible control for (SOC) where $X^{w,\e}_t$ is the associated trajectory. With this particular choice, the associated cost functional can be rewritten as
\begin{equation*}
J^S(\bm\alpha)=\int_0^T\int_{\R^d} \left(f(t,x,\mu^{w,\e}_t)+\psi(w(t,x))\right)\mu^{w,\e}_t(\de x) \de t +\int_{\R^d}g(x,\mu^{w,\e}_T)\mu^{w,\e}_T(\de x)=J(\mu^{w,\e},w),
\end{equation*}
where $\mu^{w,\e}_t$ is the law $\Law(X^{w,\e}_t)$ of the controlled trajectory $X^{w,\e}_t$. Then, by the strict minimality property of $\hat{\alpha}^\e_t$ for (SOC), it holds
\begin{align*}
J(\mu^\e,u^\e)=&\int_0^T\int_{\R^d} \left(f(t,x,\mu^\e_t)+\psi(u^\e(t,x))\right)\mu^\e_t(\de x) \de t +\int_{\R^d}g(x,\mu^\e_T)\mu^\e_T(\de x)\\
=&\E\left[ g(X^\e_T,\mathcal{L}(X^\e_T))+\int_0^T \left(f(t,X^\e_t,\mathcal{L}(X^\e_t))+\psi(\hat{\alpha}_t^\e)\right)\de t \right]\\
<&\E\left[ g(X^{w,\e}_T,\mathcal{L}(X^{w,\e}_T))+\int_0^T \left(f(t,X^{w,\e}_t,\mathcal{L}(X^{w,\e}_t))+\psi(\alpha_t)\right)\de t \right]\\
=&\int_0^T\int_{\R^d} \left(f(t,x,\mu^{w,\e}_t)+\psi(w(t,x))\right)\mu^{w,\e}_t(\de x) \de t +\int_{\R^d}g(x,\mu^{w,\e}_T)\mu^{w,\e}_T(\de x),
\end{align*}
for any admissible control $w$. This proves the optimality of $u^\e$ and concludes the proof.
\end{proof}

\subsection{Convergence lemmas}
In this section we prove a series of useful convergence estimates, that will be the key tools to prove our main theorem.

\begin{lem}\label{lem:comp-lip-maps}
Let $K\subset\R^d$ be a bounded set and $C_K>0$ a fixed constant. Define
\begin{equation}\label{lem:weak-comp-bochner}
    \mathcal{A}_K:=\{ u\in L^2((0,T);W^{1,\infty}(K,U)):\,\sup_{t\in (0,T)}\|u(t,\cdot)\|_{W^{1,\infty}(K,U)}\leq C_K \}.
\end{equation}
Then, $\mathcal{A}_K$ is compact in the weak $L^2((0,T);W^{1,p}(K,U))$-topology for any $p\in(1,\infty)$.
\end{lem}
\begin{proof}
See e.g. \cite[Theorem 2.5]{MFOC}
\end{proof}

We have the following convergence result for the sequence $u^\e$ defined in \eqref{def:ottimo-viscoso}.

\begin{cor}[Convergence of the controls]\label{lem:esistenza_controllo}
Let $u^\e$ be the sequence of optimal controls given by \eqref{def:ottimo-viscoso}. Then, there exist a sub-sequence, which we do not relabel, and a map $u\in L^\infty((0,T);W^{1,\infty}(\R^d,U))$ such that, for every $1< p<\infty$, the following convergence holds
\begin{equation}\label{conv-controlli}
u^\e\weakto u\,\,\mbox{ in }L^2((0,T);W^{1,p}_\mathrm{loc}(\R^d,U)), \qquad \mbox{as } \e\to 0.
\end{equation}
\end{cor}
\begin{proof}
The result is a direct application of Lemma \ref{lem:comp-lip-maps} together with Lemma \ref{lem:def-ottimo-viscoso}. 
The constant $C_K$ appearing in Lemma \ref{lem:comp-lip-maps} is chosen as $\max\{L_\lambda,R\}$ where $L_\lambda$ is the constant in \eqref{def:costante-lip-ottimo-vis} and $R$ is defined in \ref{ass:setofcontrols}. In particular, the constant $C_K$ does not depend on $\e$.
\end{proof}
We now show the convergence of the optimal trajectories.
\begin{lem}[Convergence of the trajectories]\label{lem:conv_soluzioni}
Let $u,u^\e$ be given by Corollary \ref{lem:esistenza_controllo} 
and let $\mu,\mu^\e$ be the unique solution of the deterministic \eqref{eq:c} and the viscous equation \eqref{eq:ad} with vector field $b$, and control $u,u^\e$ respectively. It then holds
\begin{equation}\label{conv:we}
\lim_{\e\to 0}\sup_{t\in(0,T)}W_2(\mu^\e_t,\mu_t)=0.
\end{equation}
\end{lem} 
\begin{proof}
We divide the proof in two steps.\\
\\
\underline{\em Step 1} \hspace{0.3cm} {\bf Compactness of the sequence $\mu^\e$.}
We start by proving compactness of $\{\mu^\e\}_{\e>0}$ in $C([0,T];\Pn)$ as a consequence of Ascoli-Arzelà's Theorem. First of all, we exploit a uniform bound on the second moment of $\mu_t^\e$: since $\mu^\e_t\in\Pn$, one can use $|x|^2$ as a test function in the equation, obtaining
\begin{align*}
    \int_{\R^d}|x|^2\mu^\e_t(\de x)&=\int_{\R^d}|x|^2\mu_0(\de x)+2\int_0^t\int_{\R^d}u^\e(s,x)\cdot x\,\mu^\e_s(\de x)\de s\\
    &+2\int_0^t\int_{\R^d}b(s,x,\mu^\e_s)\cdot x\,\mu^\e_s(\de x)\de s+2\e t.
\end{align*}
For the term involving the control, from \ref{ass:setofcontrols} we easily get
\begin{align*}
    2\left|\int_0^t\int_{\R^d}u^\e(s,x)\cdot x\mu^\e_s(\de x)\de s\right|&\leq 2R\int_0^t\int_{\R^d}|x|\mu^\e_s(\de x)\de s\\
    &\leq 2RT+2R\int_0^t\int_{\R^d}|x|^2\mu^\e_s(\de x)\de s.
\end{align*}
On the other hand, from \ref{ass:vf2} and Young inequality we get
\begin{align*}
    2\left|\int_0^t\int_{\R^d}b(s,x,\mu^\e_s)\cdot x\mu^\e_s(\de x)\de s\right|&\leq 2M\int_0^t\int_{\R^d}|x|\mu^\e_s(\de x)\de s+2M\int_0^t\int_{\R^d}|x|^2\mu^\e_s(\de x)\de s\\
    &+\int_0^t\left(\int_{\R^d}|x|\mu^\e_s(\de x)\right)^2\de s\\
    &\leq 2MT+5M\int_0^t\int_{\R^d}|x|^2\mu^\e_s(\de x)\de s. 
\end{align*}
Thus, being $0<\e<1$, we obtain
\begin{align*}
    \int_{\R^d}|x|^2\mu^\e_t(\de x)&\leq \int_{\R^d}|x|^2\mu_0(\de x) + 2(1+M+R)T\\
    &+(2R+5M)\int_0^t\int_{\R^d}|x|^2\mu^\e_s(\de x)\de s,
\end{align*}
and then Gronwall's lemma gives that
\begin{equation}\label{est:M2eps}
    \sup_{t\in(0,T)}\int_{\R^d}|x|^2\,\mu^\e_t(\de x)\leq\left[M_2(\mu_0)+2(1+M+R)T \right]e^{(2R+5M)T},
\end{equation}
providing a uniform bound on $M_2(\mu^\e_t)$. This means that the sequence $\{\mu^\e\}_{\e>0}$ takes values in a relatively compact set in $\Pn$ (endowed with $W_2$). Next, we show that the family $\{ \mu^\e\}_{\e>0}$ is equi-continuous in $C([0,T];\Pn)$. Let $X^\e_t$ be a solution of \eqref{eq:SDE-generica} with law $\mu^\e_t$, then by Proposition \ref{prop:w2-stoc} we have that
$$
W_2(\mu^\e_t,\mu^\e_s)^2 \leq \E\left[|X^\e_t-X^\e_s|^2\right].
$$
By using equation \eqref{eq:SDE-generica} we compute
\begin{align*}
    W_2(\mu^\e_t,\mu^\e_s)^2 & \leq \E\left[|X^\e_t-X^\e_s|^2\right]\\
    &\leq 4T\int_s^t\E\left[ |b(\tau, X^\e_\tau,\mu^\e_\tau)|^2+|\alpha_\tau^\e|^2\right]\de\tau +4\e\E\left[|W_t-W_s|^2\right]\\
    &\leq 4T\left(M^2\int_s^t[1+|X^\e_\tau|^2+M_2(\mu^\e_\tau)^2]\de \tau+ R^2|t-s|+c(d)|t-s|\right)\\
    &\leq C(\mu_0,T,M,R)|t-s|,
\end{align*}
which implies equi-continuity of the sequence $\{ \mu^\e\}_{\e>0}$ in $C([0,T]; \Pn )$. Since $\Pn$ is a complete metric space \cite{villani1}, by Ascoli-Arzelà's Theorem (see \cite[Proposition 3.3.1]{AGS}) the sequence $\{\mu^\e\}_{\e>0}$ is relatively compact in $C([0,T];\Pn)$ for every $T>0$. Then, up to a sub-sequence that we do not relabel, there exists a probability measure $\rho\in C([0,T];\Pn)$ such that
\begin{equation}\label{weak2}
\mu^\e\to \rho \,\mbox{ in } C([0,T];\Pn),
\end{equation}
which means that
\begin{equation}\label{conv:w2nu}
    \lim_{\e\to 0}\sup_{t\in(0,T)}W_2(\mu^\e_t,\rho_t)=0.
\end{equation}

\underline{\em Step 2} \hspace{0.3cm} {\bf Identification of the limit.}
In this step we show that $\rho$ is a solution of \eqref{eq:c}. This will imply that, by uniqueness, it holds $\rho=\mu$ and the whole sequence $\mu^\e$ converges to $\mu$. Let $\varphi\in C^\infty_c([0,T)\times\R^d)$: by Definition \ref{def:weak-sol-ad} we have that
$$
\int_0^T\int_{\R^d}\Big(\partial_t\varphi(t,x)+(b(t,x,\mu^\e_t)+u^\e(t,x))\cdot\nabla\varphi(t,x)+\e\Delta\varphi(t,x)\Big)\mu^\e_t(\de x)\de t=\int_{\R^d}\varphi(0,x)\mu_0(\de x).
$$
Notice that by Proposition \ref{prop:conv_w} we know that \eqref{weak2} implies weak convergence, thus it holds
$$
\lim_{\e\to 0} \int_0^T\int_{\R^d}\Big(\partial_t\varphi(t,x)+\e\Delta\varphi(t,x)\Big)\mu^\e_t(\de x)\de t=\int_0^T\int_{\R^d}\partial_t\varphi(t,x)\rho_t(\de x)\de t.
$$
We now consider the term involving the control. Denote by $K:=\supp(\varphi)$ and \\ $C_K:=\|u^\e\|_{L^\infty((0,T);L^\infty(K))}$. It then holds
$$
\Lip(\varphi(t,\cdot) u^\e(t,\cdot))\leq L \|\varphi\|_{L^\infty}+C_K\|\nabla\varphi\|_{L^\infty}:=C_\varphi.
$$
Then, by Lemma \ref{lem:Kant-Rub}, it holds
\begin{equation}
    \limsup_{\e}\left| \int_0^T\int_{\R^d}\varphi(t,x) u^\e(t,x) \Big(\mu^\e_t(\de x)-\rho_t(\de x)\Big)\de t \right|\leq C_\varphi\limsup_\e \int_0^T W_1(\mu^\e_t,\rho_t)\de t,
\end{equation}
which converges to $0$ as $\e\to 0$, by using \eqref{conv:w2nu} and \eqref{p-ordered}. On the other hand, it holds
$$
\lim_{\e\to 0}\int_0^T\int_{\R^d}\varphi(t,x) u^\e(t,x)\rho_t(\de x)\de t=\int_0^T\varphi (t,x)u(t,x)\rho_t(\de x)\de t,
$$
due to convergence in \eqref{conv-controlli} and the fact that $\varphi\rho$ belongs to $L^\infty((0,T);W^{-1,p'}(\R^d,\R^d))$ and has compact support. Then, we have shown
$$
\lim_{\e\to 0} \int_0^T\int_{\R^d}u^\e(t,x)\cdot\nabla\varphi(t,x)\mu^\e_t(\de x)\de t=\int_{\R^d}u(t,x)\cdot\nabla\varphi(t,x)\rho_t(\de x)\de t.
$$
We are left to prove convergence in the non-linear term: on one hand the convergence \eqref{weak2} implies that
\begin{equation}
\lim_{\e\to 0}\int_0^T\int_{\R^d}b(t,x,\rho_t)\cdot\nabla\varphi(t,x)\mu^\e_t(\de x)\de t=\int_0^T\int_{\R^d}b(t,x,\rho_t)\cdot\nabla\varphi(t,x)\rho_t(\de x)\de t.
\end{equation}
On the other hand, the uniform Lipschitz assumption \ref{ass:vf3} on $b$ implies that
\begin{equation}
\left|\int_0^T\int_{\R^d} \left(b(t,x,\mu^\e_t)-b(t,x,\rho_t)\right)\cdot\nabla\varphi(t,x)\mu^\e_t(\de x)\de t\right|\leq L\|\nabla\varphi\|_{L^\infty} \sup_{t\in (0,T)}W_2(\mu^\e_t,\rho_t),
\end{equation}
which converges to $0$ when $\e\to 0$ thanks to the convergence in \eqref{conv:w2nu}. Hence, $\rho_t$ is a solution of \eqref{eq:c} and this concludes the proof.
\end{proof}

\begin{lem}[Convergence of the cost] \label{lem:convergenza_costo}
Let $u,u^\e$ be given by Lemma \ref{lem:esistenza_controllo}
and let $\mu,\mu^\e$ be the corresponding unique solution of the deterministic \eqref{eq:c} and the viscous equation \eqref{eq:ad} with vector field $b$, and control $u,u^\e$ respectively. Let $w$ be an admissible control for $\Ppe$: then we have
\begin{equation}
\lim_{\e\to 0} J(\mu^\e,w)=J(\mu,w).
\end{equation}
Moreover, if $u^\e$ is the sequence of optimal controls as in \eqref{lem:esistenza_controllo}, we have that
\begin{equation}
J(\mu,u)\leq\liminf_{\e\to 0} J(\mu^\e,u^\e).
\end{equation}
\end{lem}
\begin{proof}
We divide the proof in two steps.\\
\\
\underline{\em Step 1} \hspace{3mm} {\bf Convergence for a fixed control.}
First, convergence of the control cost immediately follows from Lemma \ref{lem:conv_soluzioni}. We now analyze the running cost, the same argument also applies to the final cost. By \ref{ass:cost}, we have
\begin{equation}
\left|\int_0^T\int_{\R^d} \left(f(t,x,\mu^\e_t)-f(t,x,\mu_t)\right)\mu^\e_t(\de x)\de t\right|\leq L T \sup_{t\in (0,T)}W_2(\mu^\e_t,\mu_t),
\end{equation}
and the conclusion follows from Lemma \ref{lem:conv_soluzioni}.

\underline{\em Step 2} \hspace{3mm} {\bf Semi-continuity.}
It follows from \cite[Theorem 2.12]{MFOC}: arguing as in Step 1, we can show convergence of both running and final costs. Then, we must show that the control cost is lower semi-continuous with respect to the weak convergence \eqref{conv-controlli}. First of all, by Theorem \ref{thm:well_pos_cont} we can fix $r>0$ such that $\supp\,\mu_t\subset B_r$ for all $t\in[0,T]$.  Let $p>d$ and define the functional $S^\mu:L^2((0,T);W^{1,p} (B_r))\to[0,+\infty]$ as
\begin{equation}
S^\mu(g):=\begin{cases}\displaystyle
\int_0^T\int_{\R^d}\psi(g(t,x))\mu_t(\de x)\de t, \hspace{5mm}\mbox{if }\Lip(g(t,\cdot))\in L^\infty(0,T), \\
+\infty\hspace{1cm}\mbox{otherwise}.
\end{cases}
\end{equation}
By convexity of $\psi$, it is immediate to check that $S^\mu$ is convex: thus, it is sufficient to show that it is lower semi-continuous in the strong topology $L^2((0,T);W^{1,p}(B_r))$ to obtain weak lower semi-continuity. Let $g_k$ be a sequence in $L^2((0,T);W^{1,p}(B_r))$ strongly converging to some $g$. By using \ref{ass:controlcost}, we have
\begin{eqnarray*}
&&\left| \int_0^T\int_{\R^d} \left(\psi(g_k(t,x))-\psi(g(t,x))\right)\right.\left.\mu_t(\de x) \de t \right|\leq \int_0^T\int_{\R^d} \left|\psi(g_k(t,x))-\psi(g(t,x)) \right| \mu_t(\de x)\de t\\
&&\leq L\int_0^T\int_{\R^d}\left(|g_k(t,x)|+|g(t,x)|\right) \left|g_k(t,x)-g(t,x) \right| \mu_t(\de x)\de t\\
&& \leq C L\int_0^T\int_{\R^d}\left|g_k(t,x)-g(t,x) \right|^2 \mu_t(\de x)\de t \leq C L\int_0^T\|g_k(t,\cdot)-g(t,\cdot)\|_{L^\infty}^2\\
&&\leq C L\int_0^T\|g_k(t,\cdot)-g(t,\cdot)\|_{W^{1,p}}^2,
\end{eqnarray*}
where the constant $C$ depends on Sobolev embeddings and the $L^2W^{1,p}$ norm of $g_k, g$. Therefore, it holds
\begin{equation}
|S^\mu(g_k)-S^\mu(g)|\leq CL\|g_k-g\|_{L^2W^{1,p}}^2,
\end{equation}
which gives continuity with respect to  the strong topology. Thus, $S^\mu$ is weakly lower semi-continuous and by using Corollary \ref{lem:esistenza_controllo} we obtain that
\begin{equation}\label{e-liminf}
\int_0^T\int_{\R^d} \psi(u(t,x))\mu_t(\de x)\de t\leq \liminf_{\e\to 0}\int_0^T\int_{\R^d} \psi(u^\e(t,x))\mu_t(\de x)\de t.
\end{equation}
Finally, observe that $\psi$ is Lipschitz, since it is $C^1$ on the compact set $U$. We denote by $L_{\psi,U}$ its Lipschitz constant on $U$. Moreover, $u^\e$ is Lipschitz, with a Lipschitz constant $L_\lambda$ independent on $\e$, as shown in \eqref{def:costante-lip-ottimo-vis}. Then, it holds:
\begin{align*}
&\left| \int_0^T\int_{\R^d} \psi(u^\e(t,x)) \left(\mu^\e_t(\de x)-\mu_t(\de x)\right)\de t \right|\leq L_{\psi,U} L_\lambda \int_0^T W_1 (\mu^\e_t,\mu_t)\de t\\
&\leq L_{\psi,U} L_\lambda \int_0^T W_2 (\mu^\e_t,\mu_t)\de t  \leq C \sup_{t\in (0,T)} W_2(\mu^\e_t,\mu_t).
\end{align*}
Merging it with \eqref{e-liminf} and recalling \eqref{conv:w2nu}, we have 
\begin{equation}\label{e-liminf2}
\int_0^T\int_{\R^d} \psi(u(t,x))\mu_t(\de x)\de t\leq \liminf_{\e\to 0}\int_0^T\int_{\R^d} \psi(u^\e(t,x))\mu^\e_t(\de x)\de t.
\end{equation}
\end{proof}

\subsection{Proof of the main Theorem} We are now ready to prove our main result.
\begin{proof}[Proof of Theorem \ref{thm:main}]
Let $\Lambda(T,L)$ be the constant defined in \eqref{def:Lambda}. Then, by Lemma \ref{lem:def-ottimo-viscoso} we know that there exists a unique $(\mu^\e,u^\e)$ optimal pair for $\Ppe$. By Lemma \ref{lem:esistenza_controllo}, we know that there exists a function $u\in L^\infty((0,T);\Lip(\R^d,U))$ such that
$$
u^\e\weakto u \hspace{0.3cm} \mbox{in }  L^2((0,T);W^{1,p}_\mathrm{loc}(\R^d,U)),
$$
for every $1\leq p<\infty$. This proves point $(i)$ of the theorem. Moreover, there exists a unique $\mu\in C([0,T];\Pn)$ which solves \eqref{eq:cont} with control $u$; thus, by Lemma \ref{lem:conv_soluzioni} we have
$$
\mu^\e\to \mu \hspace{0.3cm} \mbox{in } C([0,T],\Pn).
$$
This is point $(ii)$ of the theorem. The convergence of the cost is a consequence of Lemma \ref{lem:convergenza_costo}. We only need to show optimality of $(\mu,u)$ for $\Pp$. Let $w\in \mathcal{A}\setminus\{u\}$ be an admissible control for $\Pp$ and $\mu^w$ the corresponding trajectory. We define $\mu^{w,\e}$ to be the unique solution of \eqref{eq:ad} with control $w$; since $(\mu^{w,\e},w)$ is an admissible pair for $\Ppe$ and $(\mu^\e,u^\e)$ is the unique optimal pair, we have that
\begin{equation}\label{a}
J(\mu^\e,u^\e) < J(\mu^{w,\e},w).
\end{equation}
Now, by Lemma \ref{lem:conv_soluzioni} we know that $\mu^{w,\e}$ converges to the unique solution $\mu^w$ of \eqref{eq:cont} with control $w$ and the associated cost converges:
$$
\lim_{\e\to 0}J(\mu^{w,\e},w)=J(\mu^w,w).
$$ 
Then, combining Lemma \ref{lem:convergenza_costo}, equation \eqref{a}, and the convergence above, it holds
\begin{equation*}
J(\mu,u)\leq\liminf_{\e\to 0} J(\mu^\e,u^\e)\leq J(\mu^w,w),
\end{equation*}
for any admissible pair $(\mu^w,w)$. Then, $(\mu,u)$ is an optimal pair for $\Pp$ and the proof is complete.
\end{proof}

\subsection{A different approach} 
It will be clear now that the lower bound on $\lambda$ provides a sufficient condition for optimality in Theorem \ref{teo:ott-st-suff}. Thus, it is reasonable to ask whether this is necessary or not for convergence of the vanishing viscosity method when an optimal control of $\Ppe$ is prescribed. We now prove the Corollary \ref{main-cor}.
\begin{proof}[Proof of Corollary \ref{main-cor}]
Let $(\mu^\e,u^\e)$ be an optimal pair for $(\mathcal{P}_\e)$, our goal is to show that the Lispchitz constant of $u^\e$ is independent on $\e$. We associate to $(\mu^\e,u^\e)$ an optimal control of the problem (SOC) in the following way. First of all, note that an optimal control for (SOC) must be in closed loop feedback form. This means that $\alpha_t=\phi(t,\cdot)$ for some deterministic function $\phi$. To prove this claim, we consider an admissible control $\alpha_t$ for (SOC) and we define $X^\alpha_t$ the corresponding controlled state. Denote by $\sigma\{X^\alpha_t\}$ the $\sigma$-field generated by $X^\alpha_t$ and define
\begin{equation}
    \tilde\alpha_t:=\E[\alpha_t|X^\alpha_t],
\end{equation}
the conditional expectation of $\alpha_t$ given $\sigma\{X^\alpha_t\}$. Note that $\tilde\alpha_t$ is square integrable, thus admissible for (SOC). By the Doob-Dynkin lemma, there exists a function $e_\alpha:\R^d\to\R^d$ such that 
\begin{equation}
    \tilde\alpha_t=e_{\alpha_t}(X_t),
\end{equation}
implying that $\tilde\alpha_t$ is in closed loop feedback form. Let $\tilde X^\alpha_t$ be the controlled process corresponding to $\tilde\alpha_t$ and denote by $\mu^{\tilde\alpha}_t=\Law(\tilde X^\alpha_t)$: since any controlled process, corresponding to a square-integrable control in closed loop feedback form, induces a solution to the advection-diffusion equation \eqref{eq:ad}, $\mu^{\tilde\alpha}_t$ solves it.  Thus, if we compute the cost of $\tilde\alpha_t$, using the convexity of $\psi$ and Jensen inequality we obtain that 
$$
J(\mu^{\tilde\alpha},\tilde\alpha)=J^S(\tilde{\bm\alpha})\leq J^S(\bm\alpha).
$$
The above inequality shows that, for any given admissible control $\bm\alpha$ we can always define an admissible control in closed loop feedback form with a smaller cost. Thus, $u$ is an optimal control also for the stochastic problem (we recall that strong existence holds for \eqref{eq:SDE-generica}). Then, by Theorem \ref{thm:cond-necessaria} and Lemma \ref{lem:lipcontr}, we obtain that $u^\e$ must be of the form \ref{def:ottimo-viscoso} and therefore unique. Here we used the assumption that $\lambda>0$ but we did not require a lower bound on it. We now provide a stability estimate for the solution of \eqref{eq:FB-PMP} with constants which are independent on $\e$, which will ensure that the constant in \eqref{def:lips_constant-master-field} is  also $\e$-independent.
From \eqref{stabilityFSDE}, \ref{lem:lipcontr} and \eqref{stability_bsde} we obtain
\begin{align*}
\E\left[ \sup_{t\in(0,T)}|X_t^{\e,'}-X_t^\e|^2 \right]&\leq C_2(T,L)\left(\E\left[ |\xi'-\xi|^2 \right]+\E\left[ \int_0^T|\hat\alpha(Y^{\e,'}_t)-\hat\alpha(Y^{\e}_t)|^2\de t \right]\right)\\
&\leq C_2(T,L)\E\left[ |\xi'-\xi|^2 \right]+\frac{C_2(T,L)}{\lambda^2}\E\left[ \int_0^T|Y^{\e,'}_t-Y^\e_t|^2\de t \right]\\
&\leq C_2(T,L)\E\left[ |\xi'-\xi|^2 \right]+\frac{T C_2(T,L)}{\lambda^2}\E\left[ \sup_{t\in(0,T)}|X_t^{\e,'}-X_t^\e|^2\right]
\end{align*}
Thus, if we define $\Lambda':=\sqrt{T C_2(T,L)}$, we have that
$$
\frac{T C_2(T,L)}{\lambda^2}\leq 1,
$$
and consequently
\begin{equation}
\E\left[ \sup_{t\in(0,T)}|X_t^{\e,'}-X_t^\e|^2 \right]\leq \frac{\lambda^2 C_2(T,L)}{\lambda^2-C_2(T,L)T}\E\left[ |\xi'-\xi|^2 \right].
\end{equation}
Finally, we use again \eqref{stability_bsde} to compute
\begin{equation}
\E\left[ \sup_{t\in(0,T)}|Y^{\e,'}_t-Y^\e_t|^2\right]\leq C_4(T,L)\frac{\lambda^2 C_2(T,L)}{\lambda^2-C_2(T,L)T}\E\left[ |\xi'-\xi|^2 \right].
\end{equation}
The conclusion now follows by using Lemma \ref{lem:esistenza_master-field} and the Lemmata of Section 4.1 and 4.2 as we did for the proof of Theorem \ref{thm:main}.
\end{proof}

\section{The role of convexity hypotheses} \label{s-controesempio}
The aim of this section is to discuss the role of the convexity hypotheses {\bf (C)} for the validity of Theorem \ref{thm:main}. In particular, we show that, by relaxing the strict convexity assumption on the control cost $\psi$, then convergence of optimal controls $u^\e$ for $\Ppe$ to an optimal control $u$ of $\Pp$ is not ensured. This is the core of the counterexample that we describe in the following.

We consider the minimization problem of the functional
\begin{equation}\label{def:J-controesempio}
J(\mu,u)=\int_0^T\int_{\R}\psi(u(t,x))\mu_t(\de x)\de t
\end{equation}
where the control cost $\psi$ is $C^\infty$, positive, convex and satisfies $\psi(s)=0 $ for $s\in[-1,1]$. Note that the function $\psi$ is clearly not strictly convex. As an example, consider the $C^\infty$, not analytic function $$\phi(x)=\begin{cases}
0 &\mbox{~for~}x=0,\\
\exp(-1/|x|) &\mbox{~for~}x\neq 0.
\end{cases}$$
Then, one can build a function $\psi(x)$ as above by choosing
$$\psi(x):=\begin{cases}
0 &\mbox{~for~}x\in[0,1],\\
\displaystyle \int_0^{x-1}ds\int_0^s \phi(t)\,dt &\mbox{~for~}x>1,\\
\psi(-x)&\mbox{~for~}x<0.
\end{cases}$$
We assume that the dynamics is given by the equation
\begin{equation}\label{eq:cont_contr}
\begin{cases}
\partial_t\mu_t+\dive[u(t,x)\mu_t]=0,\\
\mu_0=\delta_0,
\end{cases}
\end{equation}
where $\delta_0$ is the Dirac delta centered in $0$. 

Set $U=[-1,1]$ and $u\in \Lip(\R;U)$ be an admissible control. By the standard Cauchy-Lipschitz Theorem, to any admissible control we can associate a unique flow $X_t$, i.e. the solution of
\begin{equation}\label{eq:ode_contr}
    \begin{cases}
    \dot{X}_t=u(X_t),\\
    X_0=x.
    \end{cases}
\end{equation}
It follows that $\mu_t=\delta_{X_t}$ is the unique solution of \eqref{eq:cont_contr} with control $u$. It is clear that any Lipschitz function $u\in \Lip(\R;[-1,1])$ is also optimal, since the corresponding cost is identically zero.

We now consider the viscous optimal control, i.e. where the dynamics is governed by the equation
\begin{equation}\label{eq:ad_contr}
\begin{cases}
\partial_t\mu^\e_t+\dive[u(t,x)\mu^\e_t]=\e\Delta\mu^\e_t,\\
\mu^\e_0=\delta_0.
\end{cases}
\end{equation}
The same observations made for the non-viscous case apply: every Lipschitz function $u\in \Lip(\R;[-1,1])$ is optimal and it is associated to a unique solution $\mu^\e$ of \eqref{eq:ad_contr}.
For each choice of $u\in \Lip(\R;[-1,1])$, one has convergence of the controls (they are $\e$-independent), convergence of the trajectories (as an easy consequence of Lemma \ref{lem:conv_soluzioni}), and convergence of the cost (as a consequence of Lemma \ref{lem:convergenza_costo}).
However, the viscous problem has other solutions for which the convergence result does not hold. An example is provided by the function $\tilde{u}(x)=\mathrm{sign}(x)$: its cost is identically zero and thus it is an optimal control. Moreover, $\tilde{u}$ is bounded and then it is associated to a unique solution $\tilde{\mu}^\e$ of \eqref{eq:ad_contr}, see \cite{MANITA2015}. Nevertheless, since $\tilde{u}$ is independent from $\e$, it does not converge to a Lipschitz optimal control of the inviscid problem. Furthermore, notice that the equation \eqref{eq:cont_contr} with vector field $\tilde{u}$ has multiple solutions, as a consequence of the non-uniqueness for the corresponding ODE \eqref{eq:ode_contr} with initial datum $x=0$. As an example, both the trajectories $x(t)=t$ and $x(t)=-t$ are solutions of the ODE in the Caratheodory sense. \\
\\

\noindent {\bf Acknowledgments:} This work is funded by the University of Padua under the STARS Grants programme CONNECT: Control of Nonlocal Equations for Crowds and Traffic models. G. Ciampa is supported by the ERC STARTING GRANT 2021 ``Hamiltonian Dynamics, Normal Forms and Water Waves" (HamDyWWa), Project Number: 101039762. Views and opinions expressed are however those of the authors only and do not necessarily reflect those of the European Union or the European Research Council. Neither the European Union nor the granting authority can be held responsible for them.

\bibliographystyle{hplain}
\bibliography{VWbib}

\begin{thebibliography}{10}

\bibitem{achdou2}
Y.~Achdou and M.~Lauri{\`e}re.
\newblock {On the System of Partial Differential Equations Arising in Mean
  Field type Control}.
\newblock {\em {D}isc. and {C}ont. {D}ynamical {S}ystems}, 35(9):3879--3900,
  2015.

\bibitem{AGS}
L.~Ambrosio, N.~Gigli, and G.~Savar{\'e}.
\newblock {\em {G}radient {F}lows in {M}etric {S}paces and in the {S}pace of
  {P}robability {M}easures}.
\newblock Lectures in Mathematics ETH Z\"urich. Birkh\"auser Verlag, 2008.

\bibitem{Antonelli93}
F.~Antonelli.
\newblock {B}ackward-{F}orward {S}tochastic {D}ifferential {E}quations.
\newblock {\em The Annals of Applied Probability}, 3(3):777--793, 1993.

\bibitem{axelrod}
R.~Axelrod.
\newblock {\em The Evolution of Cooperation: Revised Edition}.
\newblock Basic Books, 2009.

\bibitem{bardi-fischer}
M.~Bardi and M.~Fischer.
\newblock On non-uniqueness and uniqueness of solutions in finite-horizon mean
  field games.
\newblock {\em ESAIM: COCV}, 25:44, 2019.

\bibitem{active1}
N.~Bellomo, P.~Degond, and E.~Tadmor.
\newblock {\em Active Particles, Volume 1: Advances in Theory, Models, and
  Applications}.
\newblock Birkh{\"a}user, 2017.

\bibitem{bensoussan1}
A.~Bensoussan, J.~Frehse, and P.~Yam.
\newblock {\em Mean field games and mean field type control theory}, volume
  101.
\newblock Springer, 2013.

\bibitem{bensoussan2}
A.~Bensoussan, K.C.J. Sung, S.~C.~P. Yam, and S.-P. Yung.
\newblock Linear-quadratic mean field games.
\newblock {\em Journal of Optimization Theory and Applications},
  169(2):496--529, 2016.

\bibitem{BianchiniBressan}
S.~Bianchini and A.~Bressan.
\newblock Vanishing viscosity solutions of nonlinear hyperbolic systems.
\newblock {\em Ann. of Math.}, 161(2):223--342, 2005.

\bibitem{Bogachev}
V.I. Bogachev.
\newblock {\em Measure Theory, Volume 2}.
\newblock Springer-Verlag Berlin Heidelberg, 2007.

\bibitem{MFPMP}
M.~Bongini, M.~Fornasier, F.~Rossi, and F.~Solombrino.
\newblock {Mean Field Pontryagin Maximum Principle}.
\newblock {\em {Journal of Optimization Theory and Applications}}, 175:1--38,
  2017.

\bibitem{PMPWassConst}
B~Bonnet.
\newblock {A Pontryagin Maximum Principle in Wasserstein Spaces for Constrained
  Optimal Control Problems}.
\newblock {\em ESAIM COCV}, 25(52), 2019.

\bibitem{SetValuedPMP}
B.~Bonnet and H.~Frankowska.
\newblock Differential inclusions in {W}asserstein spaces: {T}he
  {C}auchy-{L}ipschitz framework.
\newblock {\em Journal of Differential Equations}, 271:594--637, 2021.

\bibitem{PMPWass}
B.~Bonnet and F.~Rossi.
\newblock {The Pontryagin Maximum Principle in the Wasserstein Space}.
\newblock {\em Calculus of Variations and Partial Differential Equations},
  58:11, 2019.

\bibitem{LipReg}
B.~Bonnet and F.~Rossi.
\newblock Intrinsic {L}ipschitz regularity of mean-field optimal controls.
\newblock {\em SIAM J. Control Optim.}, 59(3):2011 -- 2046, 2021.

\bibitem{Burger2020}
M.~Burger, R.~Pinnau, O.~Totzeck, O.~Tse, and A.~Roth.
\newblock {Instantaneous Control of Interacting Particle Systems in the
  Mean-Field Limit}.
\newblock {\em Journal of Computational Physics}, 405:109181, 2020.

\bibitem{camazine}
S.~Camazine, J.-L. Deneubourg, N.~R. Franks, J.~Sneyd, G.~Theraulaz, and
  E.~Bonabeau.
\newblock {\em {Self-Organization in Biological Systems}}.
\newblock Princeton University Press, 2001.

\bibitem{Caponigro2015}
M.~Caponigro, M.~Fornasier, B.~Piccoli, and E.~Tr{\'e}lat.
\newblock {Sparse Stabilization and Control of Alignment Models}.
\newblock {\em Math. Mod. Meth. Appl. Sci.}, 25 (3):521--564, 2015.

\bibitem{Cardaliaguet2015}
P.~Cardaliaguet, P.~Jameson~Graber, A.~Porretta, and D.~Tonon.
\newblock Second order mean field games with degenerate diffusion and local
  coupling.
\newblock {\em {N}onlinear {D}ifferential {E}quations and {A}pplications},
  22:1287--1317, 2015.

\bibitem{Carmona2013b}
R.~Carmona and F.~Delarue.
\newblock {Mean field forward-backward stochastic differential equations}.
\newblock {\em Electronic Communications in Probability}, 18:1 -- 15, 2013.

\bibitem{Carmona2015}
R.~Carmona and F.~Delarue.
\newblock {Forward-Backward Stochastic Differential Equations and Controlled
  McKean-Vlasov Dynamics}.
\newblock {\em Annals of Probability}, 43(5):2647--2700, 2015.

\bibitem{CarmonaDelarueI}
R.~Carmona and F.~Delarue.
\newblock {\em Probabilistic Theory of Mean Field Games with Applications I}.
\newblock Number~83 in Probability Theory and Stochastic Modelling. Springer,
  2018.

\bibitem{CarmonaDelarueII}
R.~Carmona and F.~Delarue.
\newblock {\em Probabilistic Theory of Mean Field Games with Applications II}.
\newblock Number~84 in Probability Theory and Stochastic Modelling. Springer,
  2018.

\bibitem{Cavagnari2018}
G.~Cavagnari, A.~Marigonda, K.T. Nguyen, and F.S Priuli.
\newblock {G}eneralized {C}ontrol {S}ystems in the {S}pace of {P}robability
  {M}easures.
\newblock {\em Set-Valued and Var. Analysis}, 26(3):663--691, 2018.

\bibitem{Cavagnari2020}
G.~Cavagnari, A.~Marigonda, and B.~Piccoli.
\newblock {Generalized Dynamic Programming Principle and Sparse Mean-Field
  Control Problems}.
\newblock {\em Journal of Mathematical Analysis and Applications},
  481(1):123437, 2020.

\bibitem{CCDelarue14}
J.~F. Chassagneux, D.~Crisan, and F.~Delarue.
\newblock A probabilistic approach to classical solutions of the master
  equation for large population equilibria.
\newblock {\em Mem. Amer. Math. Soc.}, 280(1379), 2022.

\bibitem{Chemin96}
J.-Y. Chemin.
\newblock A remark on the inviscid limit for two-dimensional incompressible
  fluids.
\newblock {\em Commun. Part. Diff. Eq.}, 21:1771--1779, 1996.

\bibitem{CCS21}
G.~Ciampa, G.~Crippa, and S.~Spirito.
\newblock {S}trong {C}onvergence of the {V}orticity for the 2{D} {E}uler
  {E}quations in the {I}nviscid {L}imit.
\newblock {\em Archive for Rational Mechanics and Analysis}, 240:295–326,
  2021.

\bibitem{LQ-CDC}
G.~Ciampa and F.~Rossi.
\newblock Vanishing viscosity for linear-quadratic mean-field control problems.
\newblock {\em IEEE 60th Annual Conference on Decision and Control (CDC)},
  pages 185--190, 2021.

\bibitem{CPT}
E.~Cristiani, B.~Piccoli, and A.~Tosin.
\newblock {\em {Multiscale Modeling of Pedestrian Dynamics}}, volume~12.
\newblock Springer, 2014.

\bibitem{Duprez2019}
M.~Duprez, M.~Morancey, and F.~Rossi.
\newblock {A}pproximate and {E}xact {C}ontrollability of the {C}ontinuity
  {E}quation with a {L}ocalized {V}ector {F}ield.
\newblock {\em SIAM Journal on Control and Optimization}, 57(2):1284--1311,
  2019.

\bibitem{Duprez2020}
M.~Duprez, M.~Morancey, and F.~Rossi.
\newblock Minimal time for the continuity equation controlled by a localized
  perturbation of the velocity vector field.
\newblock {\em Journal of Differential Equations}, 269(1):82--124, 2020.

\bibitem{FLOS}
M.~Fornasier, S.~Lisini, C.~Orrieri, and G.~Savar{\'e}.
\newblock {M}ean-{F}ield {O}ptimal {C}ontrol as {G}amma-{L}imit of {F}inite
  {A}gent {C}ontrols.
\newblock {\em Europ. Journ. of App. Math.}, 30:1153–1186, 2019.

\bibitem{FPR}
M.~Fornasier, B.~Piccoli, and F.~Rossi.
\newblock {M}ean-{F}ield {S}parse {O}ptimal {C}ontrol.
\newblock {\em Phil. Trans. R. Soc. A}, 372(2028):20130400, 2014.

\bibitem{MFOC}
M.~Fornasier and F.~Solombrino.
\newblock {Mean Field Optimal Control}.
\newblock {\em Esaim COCV}, 20(4):1123--1152, 2014.

\bibitem{GangboMeszaros}
W.~Gangbo and A.~R. M\'esz\'aros.
\newblock Global well-posedness of master equations for deterministic
  displacement convex potential mean field games.
\newblock {\em Comm. Pure Appl. Math.}, to appear. Preprint available at
  \url{https://arxiv.org/abs/2004.01660}.

\bibitem{Gangbo2015}
W.~Gangbo and A.~Swiech.
\newblock {Existence of a Solution to an Equation Arising in the Theory of Mean
  Field Games}.
\newblock {\em Journal of Differential Equations}, 259(11):6573--6643, 2015.

\bibitem{helbing}
D.~Helbing.
\newblock {\em Quantitative sociodynamics: stochastic methods and models of
  social interaction processes}.
\newblock Springer Science \& Business Media, 2010.

\bibitem{Jackson2010}
M.O. Jackson.
\newblock {\em {Social and Economic Networks}}.
\newblock Princeton University Press, 2010.

\bibitem{Kruzkov}
S.~N. Kru\v{z}kov.
\newblock First order quasilinear equations with several independent variables.
\newblock {\em Mat. Sb. (N.S.)}, 123:228--255, 1970.

\bibitem{kunita}
H.~Kunita.
\newblock {\em Stochastic differential equations and stochastic flows of
  diffeomorphisms}.
\newblock Lecture Notes in Math. Springer-Verlag, 1984.

\bibitem{numerico}
R.~J. LeVeque.
\newblock {\em Finite volume methods for hyperbolic problems}.
\newblock Cambridge Texts in Applied Mathematics. Cambridge University Press,
  Cambridge, 2002.

\bibitem{MANITA2015}
O.~A. Manita, M.~S. Romanov, and S.~V. Shaposhnikov.
\newblock On uniqueness of solutions to nonlinear
  {F}okker-{P}lanck-{K}olmogorov equations.
\newblock {\em Nonlinear Analysis}, 128:199--226, 2015.

\bibitem{MANITA2014}
O.~A. Manita and S.~V. Shaposhnikov.
\newblock Nonlinear parabolic equations for measures.
\newblock {\em St. Petersburg Math. J.}, 25:43--62, 2014.

\bibitem{golse}
A.~Muntean, J.~Rademacher, and A.~Zagaris.
\newblock {\em Macroscopic and large scale phenomena: coarse graining, mean
  field limits and ergodicity}.
\newblock Springer, 2016.

\bibitem{oksendal}
B.~{\O}ksendal.
\newblock {\em Stochastic differential equations}.
\newblock Springer, 2003.

\bibitem{PengWu99}
S.~Peng and Z.~Wu.
\newblock Fully coupled forward–backward stochastic differential equations
  and applications to optimal control.
\newblock {\em SIAM J. Control Optim.}, 37(3):825--843, 1999.

\bibitem{Pedestrian}
B.~Piccoli and F.~Rossi.
\newblock {Transport Equation with Nonlocal Velocity in {W}asserstein Spaces:
  Convergence of Numerical Schemes}.
\newblock {\em Acta App. Math.}, 124(1):73--105, 2013.

\bibitem{ControlKCS}
B.~Piccoli, F.~Rossi, and E.~Tr{\'e}lat.
\newblock {Control of the kinetic Cucker-Smale model}.
\newblock {\em SIAM J. Math. Anal.}, 47(6):4685--4719, 2015.

\bibitem{Pogodaev2016}
N.~Pogodaev.
\newblock {Optimal Control of Continuity Equations}.
\newblock {\em NoDEA}, 23:21, 2016.

\bibitem{sepulchre}
R.~Sepulchre.
\newblock Consensus on nonlinear spaces.
\newblock {\em Annual Reviews in Control}, 35(1):56--64, 2011.

\bibitem{Sznitman91}
AS. Sznitman.
\newblock {\em Topics in propagation of chaos}.
\newblock Springer-Verlag, Berlin, 1991.

\bibitem{Tao-book}
T.~Tao.
\newblock {\em An Epsilon of Room, I: Real Analysis: pages from year three of a
  mathematical blog}, volume 117 of {\em Graduate Studies in Mathematics}.
\newblock American Mathematical Society, 2010.

\bibitem{villani}
C.~Villani.
\newblock {\em Topics in optimal transportation}, volume~58 of {\em Graduate
  Studies in Mathematics}.
\newblock American Mathematical Society, Providence, RI, 2003.

\bibitem{villani1}
C.~Villani.
\newblock {\em {Optimal Transport: Old and New}}.
\newblock Springer-Verlag, Berlin, 2009.

\end{thebibliography}

\end{document}